\documentclass{amsart}
\usepackage{math}
\usepackage{tikz,mathabx,mathrsfs,mathtools,amsrefs}
\usetikzlibrary{cd,quotes,angles,decorations,arrows,automata,lindenmayersystems}
\newenvironment{fsa}[1][auto]{\begin{tikzpicture}[->,>=stealth',
    shorten >=1pt,auto,node distance=3cm,double distance between line centers=0.45ex,
    initial text=,accepting/.style=accepting by arrow,
    every state/.style={inner sep=3pt,minimum size=0pt},
    every loop/.style={looseness=12},semithick,#1]}{\end{tikzpicture}}
\pgfdeclaredecoration{single line}{initial}{
  \state{initial}[width=\pgfdecoratedpathlength-1sp]{\pgfpathmoveto{\pgfpointorigin}}
  \state{final}{\pgfpathlineto{\pgfpointorigin}}
}
\newtheorem{mainthm}{Theorem}

\newcommand\Aut{\operatorname{Aut}}
\newcommand\Mon{\operatorname{Mon}}
\def\mathvisiblespace{\text{\textvisiblespace}}
\newcount\tmpcnta

\newcount\tempn
\newcommand\recurbs[4]{\ifnum#1=0
  \fill (0,0) rectangle +(1,1);
  \else
  \foreach\x/\y in {#4} {\tempn=#1\advance\tempn by -1
    \begin{scope}[xscale=1/#2,yscale=1/#3,xshift=\x cm,yshift=\y cm]
      \recurbs{\the\tempn}{#2}{#3}{#4}
    \end{scope}
  };
  \fi}
\newcommand\bs[4]{\recurbs{#1}{#2}{#3}{#4}
  \pgfmathsetmacro\griddx{#2^(-#1)}
  \pgfmathsetmacro\griddy{#3^(-#1)}
  \draw[thin,gray!25,xstep=\griddx,ystep=\griddy] (0,0) grid +(1,1);
}
\newcommand\bsrule[2]{\begin{tikzpicture}[scale=0.3,>=stealth']
    \bs{0}{1}{1}{};
    \draw[->] (1.5,0.5) -- +(1,0);
    \begin{scope}[xshift=3cm,yshift=0.5cm-#1*0.5cm,scale=#1]
      \bs{1}{#1}{#1}{#2}
    \end{scope}
  \end{tikzpicture}}

\begin{document}
\title[Monadic second-order logic and dominoes on self-similar graphs]{Monadic second-order logic and the domino problem on self-similar graphs}
\author{Laurent Bartholdi}
\address{Mathematisches Institut, Georg-August Universit\"at zu G\"ottingen}
\email{laurent.bartholdi@gmail.com}
\date{November 5th, 2020}
\begin{abstract}
  We consider the domino problem on Schreier graphs of self-similar
  groups, and more generally their monadic second-order logic. On the
  one hand, we prove that if the group is bounded then the graph's
  monadic second-order logic is decidable. This covers, for example,
  the Sierpiński gasket graphs and the Schreier graphs of the Basilica
  group. On the other hand, we already prove undecidability of the
  domino problem for a class of self-similar groups, answering a
  question by Barbieri and Sablik, and some examples including one of
  linear growth.
\end{abstract}
\maketitle

%%%%%%%%%%%%%%%%%%%%%%%%%%%%%%%%%%%%%%%%%%%%%%%%%%%%%%%%%%%%%%%% 
\section{Introduction}
The \emph{domino problem} is (in spite of its connection to monadic
second-order logic, see~\S\ref{ss:mso}) a mockingly elementary
question to ask of an edge-labelled graph: ``given a collection of
labelled dominoes (with numbers on their ends), can one put a domino
on each edge of the graph in such a manner that edge labels and vertex
numbers match?''

This problem is clearly solvable by brute force if the graph is
finite, and it is easy to see that it is solvable if the graph is a
line. Remarkably, if the graph is the square grid (with edges labelled
vertical/horizontal) then this problem is unsolvable, as was shown by
Berger~\cite{berger:undecidability}.

This result should not be seen as negative; rather, it points to the
universal computing power present in the square grid. It is natural to
delineate, then, the frontier between decidability and undecidability,
in terms of the structure of the underlying graph.

A large source of labelled graphs worthy of study arises from group
theory: given a group $G$ with generating set $A$ and acting on a
space $X$, consider the graph with vertex set $X$, having for all
$s\in A,x\in X$ an edge labelled $s$ from $x$ to $s x$. Such graphs
are known as \emph{Schreier graphs} since their appearance
in~\cite{schreier:untergruppen}. In this setting, an instance of the
domino problem is a subset $\Theta\subseteq B\times A\times B$, and
the question is whether there exists a colouring $\tau\colon X\to B$
with $(\tau(x),s,\tau(s x))\in\Theta$ for all $s\in A,x\in X$.

\subsection{Decidability results}
Our first, ``positive'' result concerns graphs with an abundance of
local cut points. These graphs are intimately connected to
\emph{finitely ramified fractals} and \emph{bounded transducer
  automata}, see just below for definitions. This result will be
extended, in~\S\ref{ss:mso}, to decidability of the graph's monadic
second-order theory.
\begin{mainthm}[= Proposition~\ref{prop:pcfdecidable}]\label{thm:bounded}
  The domino problem is decidable on post-critically finite
  self-similar graphs.
\end{mainthm}
The model of such graphs is a discrete avatar of the Sierpiński gasket; it
is the graph with vertex set all $(m,n)\in\N^2$ such that the binomial
co\"efficient $\binom n m$ is odd and edges connecting nearest
neighbours.

A \emph{transducer} is a finite rooted graph $\Phi$ with input and
output labels in a finite set $S$ on every edge, and such that that at
every vertex and for every $s\in S$ there is a single outgoing edge
with input label $s$. The transducer produces a transformation $\phi$
of the space $S^\N$ of right-infinite strings over $S$, as follows:
given $\xi\in S^\N$, there is a unique right-infinite path, starting at
the root, and with input labels $\xi$; then $\phi(\xi)\in S^\N$ are the
output labels along this path. Fixing one graph $\Phi$ and varying its
root produces a finite collection of transformations, and if all of
them are invertible then the group of permutations of $S^\N$ that they
generate is called a \emph{self-similar group}; its generating set is
naturally in bijection with the vertex set of $\Phi$. Some very small
transducers produce rich and interesting groups,
see~\S\ref{ss:sierpinski} for an example called the ``Hanoi tower
group'' in connection with the Sierpiński gasket. The graphs that we
are interested in are Schreier graphs of self-similar groups.

The self-similar group associated with a transducer $\Phi$ is
\emph{bounded}, see~\S\ref{ss:bounded}, if the exiting arrows along
every oriented cycle in $\Phi$ all eventually lead to a vertex
representing the identity transformation. The corresponding Schreier
graphs are closely related to a self-similar compactum known under
various names in the literature: ``hierarchical fractal'', ``nested
fractal'', or ``finitely ramified fractal''. Kigami considers
in~\cite{kigami:harmonic} a compact space $K$ and a family of
self-maps $(F_s)_{s\in S}$ of $K$, such that there exists a ``coding
map'' $\pi\colon S^{-\N}\to K$ with $\pi(w s)=F_s(\pi(w))$ for all
left-infinite words $w\in S^{-\N}$ and all $s\in S$. \emph{Tiles} of
level $n$ are images of cylinders, namely $\pi(S^{-\N}v)$ for a word
$v\in S^n$. Construct the graph whose vertices are all depth-$n$
tiles, with an edge between two tiles if they intersect; and take a
limit of such graphs as $n\to\infty$. For example, the Sierpiński
gasket admits three contractions $F_1,F_2,F_3$ onto its level-$1$
tiles, and the associated graph is essentially the graph mentioned
above.

General constructions by Nekrashevych~\cite{nekrashevych:ssg}
establish a duality between certain (``contracting'') self-similar
groups and expanding self-covering maps on a compact set called
\emph{limit space}. We take in~\S\ref{ss:ss} the opportunity to
clarify the connection between Kigami's and Nekrashevych's
definitions: Nekrashevych's limit space $L$ is a quotient of Kigami's
space $K$, and the $F_s$ are branches of the self-covering of
$L$. Kigami's ``ancestor structure'', a combinatorial gizmo extracted
from $(K,S)$ and powerful enough to allow reconstruction of $(K,S)$,
may be directly produced from a transducer defining the bounded
self-similar group.

A large family of self-similar groups, and associated Schreier graphs,
arise as ``iterated monodromy groups'' of complex polynomials all of
whose critical points are eventually periodic. We shall not need the
definition of ``iterated monodromy groups''; suffice it to say that
they are the algebraic counterpart to the dynamical system afforded by
the polynomial acting on its Julia set. The associated graphs are thus
limits of simplicial approximations of the Julia set of the
polynomial. One prominent example, whose Schreier graphs have been
extensively studied~(see e.g.\
\cite{dangeli-donno-matter-nagnibeda:basilica}), is the ``Basilica
group'' associated with the polynomial $z^2-1$.

%physics: eg. gases or liquids in porous media, solid-state physics, see Feder, ``fractals''

\subsection{Undecidability results}\label{ss:introu}
Our next results are in the ``negative'' direction.  Two examples of
Schreier graphs of self-similar groups appeared to have good chances
of being close to the frontier of (un)decidability of the domino
problem: the ``long range graph'' and the ``Barbieri-Sablik
$H$-graph''.  I am grateful to Ville Salo for discussions on
translating the Barbieri-Sablik self-similar structure into a
particularly simple automatic graph. I show that, for each of them,
the domino problem is undecidable. This last graph serves to answer a
question by Barbieri and Sablik, which will be reviewed later.

The ``long range graph'', see~\S\ref{ss:longrange}, is a deterministic
model of long range percolation on the integers: nearest neighbours
are connected, and for all $s>0$ points at distance $2^s$ apart are
connected ``with probability $2^{-s}$'', but in a deterministic
manner: precisely if they belong to $2^s\Z+2^{s-1}$. The graph, and
the transducer producing it, are\\
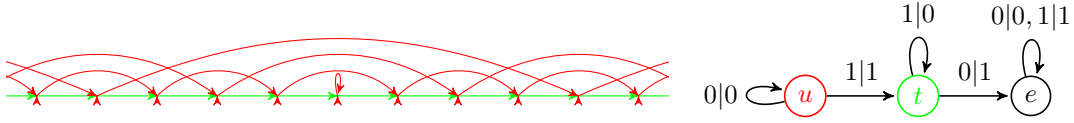
\centerline{\begin{tikzpicture}[scale=0.8,baseline,>=stealth']
    \clip (-5.5,-1) rectangle (5.5,1);
    \foreach\i in {-6,...,5} {
      \draw[->,green] (\i,0) -- +(1,0);
    }
    \foreach\i in {-7,-5,...,5} {
      \draw[->,red] (\i,0) edge[bend left=45] +(2,0);
    }
    \foreach\i in {-6,-2,...,5} {
      \draw[->,red] (\i,0) edge[bend left=36] +(4,0);
    }
    \foreach\i in {-12,-4,...,5} {
      \draw[->,red] (\i,0) edge[bend left=24] +(8,0);
    }
    \draw[->,red] (0,0) edge[loop above] ();
  \end{tikzpicture}\quad\begin{fsa}[baseline,scale=0.5]
    \node[state,green] (t) at (0,0) {$t$};
    \node[state,red] (u) at (-3,0) {$u$};
    \node[state] (e) at (3,0) {$e$};
    \path (t) edge[loop above] node {$1|0$} ()
    (t) edge node {$0|1$} (e)
    (u) edge[loop left] node {$0|0$} ()
    (u) edge node {$1|1$} (t)
    (e) edge [loop above] node {$0|0,1|1$} ();
  \end{fsa}}\\
This transducer does not generate a bounded group, but rather a
``linear growth group'': the number of paths of length $n$ in the
transducer ending at a non-trivial state is not bounded, but grows
linearly in $n$; see~\cites{sidki:acyclicity,amir-a-v:linamen}.
\begin{mainthm}[see~\S\ref{ss:longrange}]\label{thm:longrange}
  The domino problem is undecidable on the long range graph.
\end{mainthm}

The second graph is a subgraph of the half-plane, in which some edges
are replaced by loops, see~\S\ref{ss:bs}. Again the graph and the transducer:\\
\centerline{\begin{fsa}[baseline,scale=0.8]
    \node[state,green] (y) at (0,0) {$y$};
    \node[state,red] (x) at (0,2) {$x$};
    \node[state] (e) at (0,4) {$e$};
    \node[state,blue] (z) at (0,6) {$z$};
    \path (y) edge[out=-20,in=20,loop] node[above right=-1mm] {$00|00$} ()
    (y) edge node[right] {$10|10$} (x)
    (y) edge[bend left] node {$\begin{matrix}01|01\\11|11\end{matrix}$} (e)
    (x) edge node[right] {$\begin{matrix}10|00,11|01\\00|10,01|11\end{matrix}$} (e)
    (z) edge[out=-20,in=20,loop] node[below right=-1mm] {$\begin{matrix}01|00\\11|10\end{matrix}$} ()
    (z) edge node[left] {$\begin{matrix}00|01\\10|11\end{matrix}$} (e);
  \end{fsa}\quad\begin{tikzpicture}[scale=0.6,baseline]
    \clip (-0.5,-0.5) rectangle (16.5,8.5);
    \foreach\i in {0,...,16} \draw[blue,thin] (\i,-1) -- (\i,9);
    \foreach\i in {0,2,...,16} { \foreach\j in {0,...,8} \draw[red,thin] (\i,\j) -- +(1,0); }
    \foreach\s in {2,4,8,16} {
      \pgfmathsetmacro\threes{3*\s}
      \foreach\i in {\s,\threes,...,16} { \foreach\j in {0,\s,...,8} \draw[green,thin] (\i-1,\j) -- +(1,0); }
    }
    \foreach\i in {0,...,16} {
      \draw[green] (0,\i) .. controls +(135:0.6) and +(225:0.6) .. +(0,0);
    }
    \foreach\s in {2,4,8,16} {
      \foreach\i in {0,\s,...,16} {
        \foreach\j in {0,\s,...,16} {
          \draw[green] (\i+\s-1,\j+0.5*\s) .. controls +(45:0.6) and +(-45:0.6) .. +(0,0);
          \draw[green] (\i+\s,\j+0.5*\s) .. controls +(135:0.6) and +(225:0.6) .. +(0,0);
        }
      }
    }
  \end{tikzpicture}}

\begin{mainthm}[see~\S\ref{ss:bs}]\label{thm:bs}
  The domino problem is undecidable on the Barbieri-Sablik $H$-graph.
\end{mainthm}

\subsection{Monadic second-order logic}\label{ss:mso}
Consider an $A$-labelled graph $\Gamma$ with root $x_0$. Monadic
second-order logic is concerned with formulas built from variables
$X,Y,\dots$ representing sets of vertices in $\Gamma$, the constant
$\{x_0\}$, for all $a\in A$ an operation $a\cdot X$ representing all
vertices reachable from $X$ by following an $a$-labelled edge, the
relation $\subseteq$, and usual boolean connectives $\vee,\wedge,\neg$
and quantifiers $\forall,\exists$. The \emph{monadic second-order
  theory} $\Mon(\Gamma)$ consists of all formulas without free variables
that hold in $\Gamma$.

Note that many usual graph-theoretic notions are readily definable in
second-order logic; for example, the empty set is characterized by the
formula $\phi(X)\equiv\forall Y[X\subseteq Y]$; `$X=Y$' is shorthand
for `$X\subseteq Y\wedge Y\subseteq X$'; `$X\cap Y$' is expressed as
$\phi(Z)\equiv Z\subseteq X\wedge Z\subseteq Y\wedge\forall
W[W\subseteq X\wedge W\subseteq Y\Rightarrow W\subseteq Z]$;
singletons are characterized by
$\phi(X)\equiv X\neq\emptyset\wedge\forall Y[Y\subset X\Rightarrow
Y=\emptyset\vee Y=X]$, using which one can represent vertices and
write `$x\in X$' to mean `$\{x\}\subseteq X$'; etc. As an example of
the power of this logic, the graph $\Gamma$ is connected if and only
if
$\forall X[X=\emptyset\vee X=\Gamma\vee\bigvee_{a\in A}a\cdot X\neq
X]$. We refer to~\cite{muller-schupp:mso}*{\S3} for details.

More fundamentally, an instance $\Theta\subseteq B\times A\times B$ of
the domino problem is easily translated to the sentence
\[\exists X_b(b\in B)\bigg[\bigsqcup_{b\in B}X_b=\Gamma\wedge\bigwedge_{(b,a,b')\not\in\Theta}X_{b'}\cap a\cdot X_b=\emptyset\bigg];
\]
and for a ``seeded'' (see~\S\ref{ss:seeded}) domino problem
$(\Theta,b_0)$ one adds the clause `$x_0\in X_{b_0}$'.

The graph $\Gamma$ has \emph{decidable monadic second-order theory} if
there is an algorithm that, given a sentence in the logic of
$A$-labelled graphs, decides whether it holds in $\Gamma$. Thus if
$\Gamma$ has decidable monadic second-order theory then the domino
problem is decidable for $\Gamma$, and the domino problem is contained
in the existential fragment of monadic second-order logic. We shall
improve on Theorem~\ref{thm:bounded} as follows:
\begin{mainthm}[= Theorem~\ref{thm:msodecidable}]
  The monadic second-order theory of a post-critically finite
  self-similar graph is decidable.
\end{mainthm}

\subsection{Barbieri and Sablik's self-similar structures}
Barbieri and Sablik, in~\cite{barbieri-sablik:ssdomino}, consider the
domino problem on self-similar structures. Their definition is tightly
connected to the Euclidean grid: they consider a black/white colouring
of the grid defined by iterating a substitution. They then consider
domino problems on the grid, but for which the adjacency of tiles is
only enforced on black cubes. For example, in dimension $2$, consider
the substitution\\
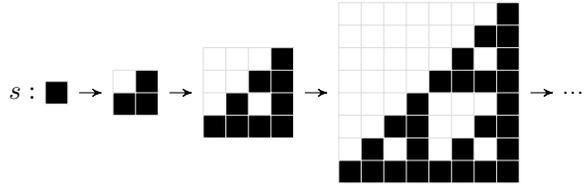
\centerline{\begin{tikzpicture}[scale=0.3,>=stealth']
    \node[anchor=east] at (0,0.5) {$s:$};
    \foreach\n/\x/\s in {0/0/1,1/3/2,2/7/4,3/13/8} {
      \begin{scope}[xshift=\x cm,yshift=0.5cm-0.5*\s cm,scale=\s]
        \bs{\n}{2}{2}{0/0,1/0,1/1};
      \end{scope}
      \draw[->] (\x+\s+0.5,0.5) -- +(1,0);
    }
    \draw[dotted,thick] (23,0.5) -- +(1,0);
  \end{tikzpicture}}\\
It produces essentially the same graph as the one associated with the
Sierpiński gasket, by colouring the plane via a limit of $s^n$ and
considering the graph with one vertex per black square and an edge
between touching squares.

More formally, fixing a substitution $s$ as above, they consider the
following variant of the domino problem on $\Z^d$: ``Given a set of
colours $B=\{\circ\}\sqcup B_\bullet$ and tileset
$\Theta\subseteq B\times\{-1,0,1\}^d\times B$, is it possible for all $n\in\N$ to tile
$s^n(\tikz{\fill(0,0) rectangle (1.5ex,1.5ex);})$ using $\Theta$ in such a
manner that white boxes are coloured $\circ$ and black boxes
$B_\bullet$?'' (this definition is a slight variant of theirs, and is equivalent if the black boxes are connected.)

They separate substitutions into ``bounded connectivity'',
``isthmus'', and ``grid'' type, according to the number of black paths
crossing $s(\tikz{\fill(0,0) rectangle (1.5ex,1.5ex);})$, and show that
in the ``bounded connectivity'' case the domino problem is decidable,
while in the ``grid'' case the domino problem is undecidable. They
leave open the ``isthmus case'', for which a prototypical substitution
is\\
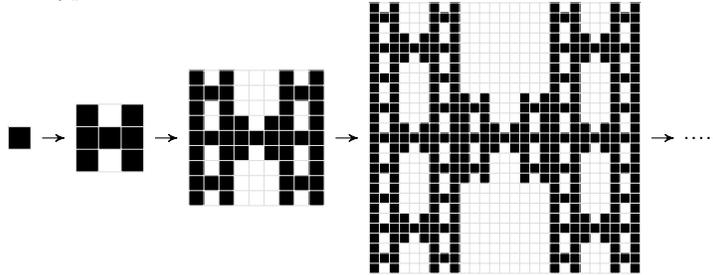
\centerline{\begin{tikzpicture}[scale=0.3,>=stealth']
    \foreach\n/\x/\s in {0/0/1,1/3/3,2/8/6,3/16/12} {
      \begin{scope}[xshift=\x cm,yshift=0.5cm-0.5*\s cm,scale=\s]
        \bs{\n}{3}{3}{0/0,0/1,0/2,1/1,2/0,2/1,2/2};
      \end{scope}
      \draw[->] (\x+\s+0.5,0.5) -- +(1,0);
    }
    \draw[dotted,thick] (30,0.5) -- +(1.5,0);
  \end{tikzpicture}}\\

The right half of the corresponding graph is essentially the $H$-graph
mentioned in Theorem~\ref{thm:bs}. Using this, we prove:
\begin{thm}[See~\S\ref{ss:isthmus}]
  The domino problem associated with a substitution $s$ is undecidable
  if $s$ contains an \emph{isthmus}: a certain configuration of blocks
  forming at least two strips in one direction and one in another.
\end{thm}

\subsection{Some conjectures and remarks}
\emph{In fine}, all proofs of undecidability of the domino problem, or
more generally of the monadic second-order logic of a graph, seem to
rely on ``space-time diagrams'': there are subsets $C_0,C_1,\dots$ of
the graph on which the stateset of a machine (be it a Turing machine,
or one of Kari's piecewise-affine machines~\cite{kari:revisited}) can
be represented; and there are enough connections in the graph between
$C_t$ and $C_{t+1}$ so that the one-step evolution of the machine can
be logically enforced. If the $C_t$ are actually subgraphs and the
machine's state is represented by a bi-infinite tape, then each $C_t$
is a copy of $\Z$ and the space-time diagram is a copy of
$\Z\times\N$. (Note that there are undecidable problems that do not
reduce to undecidability of the halting problem --- one speaks of
Turing degrees strictly between $\mathbf0$ and $\mathbf{0'}$,
see~\cite{muchnik:intermediate} --- but I am not aware of any natural
such example, a fortiori as a tiling problem.)

In terms of Schreier graphs, this means that if a subgroup $\Z^2\le G$
acts freely on an orbit $G\cdot\xi$ then the domino problem on the
corresponding Schreier graph is undecidable; and much more general
statements are true. Following~\cite{jeandel:translation}, if $G$
contains a direct product $H_1\times H_2$ of two infinite, finitely
generated groups, and each $H_i$ with $i=1,2$ has all orbits infinite
in its action on $H_{3-i}\backslash(G\cdot\xi)$, the space of orbits
of the other, then the domino problem on the Schreier graph of
$G\cdot\xi$ is also undecidable.

In~\cite{bartholdi-salo:ll} we consider the domino problem on a Cayley
graph $\Gamma$ which does not contain any grid, the ``lamplighter
group'' $\Z/2\wr\Z$, and show that nevertheless its ``seeded''
(see~\S\ref{ss:seeded}) domino problem is undecidable. The main,
general idea is that an auxiliary domino problem may be used to mark
some vertices and some sequences of edges to simulate a grid within
$\Gamma$. (In fact, it would be equally good to simulate any graph
with unsolvable domino problem, but somehow we always fall back on the
grid). This is the argument used in~\S\ref{ss:undecidable} to prove
Theorems~\ref{thm:longrange} and~\ref{thm:bs}; though we do not make
use of the general results of~\cite{bartholdi-salo:ll}, rather
repeating the argument in each specific case.

It follows from Seese's theorem~\cite{seese:mso} that the monadic
second-order theory of a graph $\Gamma$ is undecidable if it has
unbounded \emph{treewidth} (see~\S\ref{ss:treewidth}); equivalently, if
$\Gamma$ contains arbitrarily large grids as minors. On the other
hand, Ville Salo pointed out to me that, if $\Gamma$ contains
sufficiently sparse grids then the domino problem may be
decidable. For concreteness, consider a mutilated grid $\Z\times\N$ in
which, at height $(j-1)!+1,\dots,j!$, the horizontal edges wrap in
cycles of length $2^j$; then a tileset tiles this graph if and only if
it tiles the plane periodically. This kind of phenomenon does not seem
to be possible for Schreier graphs of self-similar groups:
\begin{conj}
  For $G$ a contracting self-similar group, the following are equivalent:
  \begin{enumerate}
  \item the domino problem is decidable on all Schreier graphs of $G$;
  \item the monadic second-order theory is decidable on all Schreier
    graphs of $G$;
  \item the the limit space of $G$ is finitely ramified.
  \end{enumerate}
\end{conj}
(We may take ``finitely ramified'' as meaning there is a discrete set
of local cut points. For example, if $G$ is conjugate to a bounded
transducer group, then this will be the case. Conversely, I suspect
that, if a digit tile is a post-critically finite fractal then there
exists a bounded group realizing it.)

\subsection{Acknowledgments}
I am grateful to Bruno Courcelle and Ville Salo for helpful comments
and generous replies to my (sometimes obscure or naive) questions.

%%%%%%%%%%%%%%%%%%%%%%%%%%%%%%%%%%%%%%%%%%%%%%%%%%%%%%%%%%%%%%%% 
\section{The domino problem on graphs}
A \emph{graph} is a pair $\Gamma=(V,E)$ of sets called \emph{vertices}
and \emph{edges}, with for every $e\in E$ a \emph{head} and
\emph{tail} $e^+,e^-\in V$. An \emph{unoriented graph} has,
furthermore, an involution $e\mapsto e'$ on $E$ such that
$(e')^\pm=e^\mp$. For a finite set $A$ of labels, an
\emph{$A$-labelled graph} is a graph endowed with a labelling
$\lambda\colon E\to A$ of its edges. A labelling is \emph{proper} if
no two edges have the same label and tail.  By contrast, for a finite
set $B$ of colours, a \emph{$B$-colouring} is a map $V\to B$, namely a
colouring of the $\Gamma$'s vertices.

The basic example of labelled graph we have in mind is a
\emph{Schreier graph}: for a finitely generated group
$G=\langle A\rangle$ acting on a set $X$, consider the graph with
vertex set $X$ and edge set $A\times X$, with $(a,x)^-=x$ and
$(a,x)^+=a\cdot x$ and $\lambda(a,x)=a$. Note that this defines a
properly labelled graph. By extension, if $\Gamma$ is a properly
$A$-labelled graph, we write $a\cdot x$ for the head of the edge
labelled $a$ with tail $x$, if it exists. The \emph{Cayley graph} is
the Schreier graph of a group acting on itself by left-translation.

If furthermore $A=A^{-1}$ is symmetric, then the Schreier graph is
unoriented, with $(a,x)'=(a^{-1},a\cdot x)$. Consider for example the
group $G=\Z^2$ acting on itself and generated by
$\{(0,\pm1),(\pm1,0)\}$; then the corresponding Schreier graph is the
usual square grid.

The \emph{domino problem} for an $A$-labelled graph
$\Gamma=(V,E,\lambda)$ is the following decision problem: \emph{given
  a finite $A$-labelled graph $\Delta$, does there exist a graph
  morphism $\Gamma\to\Delta$?}

Thus an instance of the domino problem is a finite set $B$ (the vertex
set of $\Delta$) and a subset $\Theta$ of $B\times A\times B$ (the
edges of $\Delta$, identified by their initial vertex, label and final
vertex). The output should be ``yes'' if there exists a $B$-colouring
$\tau\colon V\to B$ of $\Gamma$'s vertices such that for every edge
$e$ of $\Gamma$ one has
$(\tau(e^-),\lambda(e),\tau(e^+))\in\Theta$. We refer to $\Theta$ as a
\emph{tileset}, and to the valid colouring $\tau\colon V\to B$ as a
\emph{tiling}.

The reader may already be familiar with ``Wang tiles''; these are
squares with colours written on their four sides, and the classical
domino problem in the plane is to determine, for a given set of Wang
tiles, whether they can be used to cover the plane with matching
colours. Let us connect this formalism with the above definition.

Formally, a set of Wang tiles, for a given set of colours $C$, is a
subset $W\subseteq C^{\{S,E,N,W\}}=C^4$, and $W$ \emph{tiles} if there
exists a map $\tau\colon\Z^2\to W$ with $\tau(m,n)_N=\tau(m,n+1)_S$
and $\tau(m,n)_E=\tau(m+1,n)_W$ for all $(m,n)\in\Z^2$. Consider the
Cayley graph $\Gamma$ of $\Z^2$ generated by $\{(0,\pm1),(\pm1,0)\}$,
and construct the graph $\Delta$ with vertex set $W$, an edge labelled
$(1,0)$ from $w$ to $w'$ (and one labelled $(-1,0)$ from $w'$ to $w$)
whenever $w_E=(w')_W$, and an edge labelled $(0,1)$ from $w$ to $w'$
(and one labelled $(0,-1)$ from $w'$ to $w$) whenever
$w_N=(w')_S$. Then $W$ tiles precisely when there exists a graph
morphism $\Gamma\to\Delta$.

By a classical argument, an instance of the domino problem may specify
legal colourings of larger subgraphs than those given by the $\Theta$
above. Let us, for simplicity, restrict ourselves to the setting of
Schreier graphs: let $G=\langle A\rangle$ be a group acting on a set
$X$. If $B$ is a given set of colours, a \emph{pattern} is an element
of $B^F$ for some finite subset $F$ of $G$; or, more precisely, for
some finite subset $F$ of the free group on $A$, since this is the
only way in which elements of a general finitely generated group $G$
may be specified. An instance of the domino problem is then a
collection $\mathscr F$ of ``forbidden'' patterns for some set $B$ of
colours, and the required output is whether there exists a colouring
$\tau\colon X\to B$ that avoids all patterns in $\mathscr F$: for
every $x\in X$ and every pattern $\pi\colon F\to B$ in $\mathscr F$,
the assignment $F\ni f\mapsto\tau(f x)\in B$ is not equal to $\pi$.
\begin{lem}\label{lem:pattern}
  The pattern formulation of the domino problem is equivalent to the
  original one.
\end{lem}
\begin{proof}
  Consider first an instance of the domino problem given by
  $\Theta\subseteq B\times A\times B$. Then $\Theta$ is a set of
  patterns: the element $(b,a,b')$ is the pattern supported on
  $\{1,a\}$ with values $b,b'$ at $1,a$ respectively. Let $\mathscr F$
  be the set of patterns associated with
  $(B\times A\times B)\setminus\Theta$; then a vertex colouring avoids
  $\mathscr F$ if and only if the graph's edges are coloured by
  $\Theta$.

  Conversely, let $\mathscr F$ be a finite collection of forbidden
  patterns. There exists $R\in\N$ such that all patterns in
  $\mathscr F$ are supported on
  $\overline A\coloneqq\{1\}\cup A\cup\cdots\cup A^R$, the ball of
  radius $R$ in $G$; set then
  \[\overline B\coloneqq \{\beta\in B^{\overline A}: \text{ for all $(\pi\colon F\to B)\in\mathscr F$ we have }\pi\neq \beta\restriction F\}.
  \]
  Let $\Theta\subseteq\overline B\times A\times\overline B$ be the set
  of $(\beta,a,\beta')$ such that $\beta(g)=\beta'(g a^{-1})$ for all
  $g\in\overline A\cap\overline A a$.

  Given a vertex colouring $\overline\tau\colon X\to\overline B$ with
  edges coloured by $\Theta$, we consider the vertex colouring
  $\tau\colon X\to B$ given by $\tau(x)=\overline\tau(x)(1)$; then
  $\tau$ avoids all patterns in $\mathscr F$. Conversely, given
  $\tau\colon X\to B$ avoiding all patterns in $\mathscr F$, define
  $\overline\tau\colon X\to\overline B$ by
  $\overline\tau(x)(g)=\tau(g x)$ for all $x\in X,g\in\overline A$;
  then all edges are coloured by $\Theta$: for $a\in A$ we have
  \[\overline\tau(x)(g)=\tau(g x)=\tau(g a^{-1}a x)=\overline\tau(a x)(g a^{-1})\]
  so $(\overline\tau(x),a,\overline\tau(a x))\in\Theta$.

  We finally check that the maps $\tau\mapsto\overline\tau$ and
  $\overline\tau\mapsto\tau$ are inverses of each other. Starting from
  $\tau$, we get
  $(\text{new}\tau)(x)=\overline\tau(x)(1)=\tau(x)$. Consider
  conversely a valid tiling $\overline\tau\colon X\to\overline B$; it
  suffices to prove $\overline\tau(x)(g)=\overline\tau(g x)(1)$ for
  all $x\in X,g\in\overline A$, since then
  $(\text{new}\overline\tau)(x)(g)=(\text{new}\overline\tau)(g
  x)(1)=\tau(g x)=\overline\tau(x)(g)$. We prove the claim by
  induction over the minimal $r\in\N$ such that $g\in A^r$, the case
  $r=0$ being trivial. For $r>0$, write $g=h a$ with $h\in A^{r-1}$ and
  $a\in A$. Then
  $\overline\tau(g x)(1)=\overline\tau(h a x)(1)=\overline\tau(a
  x)(h)$ by induction; then applying the condition $\Theta$ on the
  edge between $x$ and $a x$ gives as required
  \[\overline\tau(x)(g)=\overline\tau(a x)(g a^{-1})=\overline\tau(a x)(h)=\overline\tau(g x).\qedhere\]
\end{proof}

The general formulation involves ``patches'': a \emph{patch} $\Delta$
is a rooted, $A$-labelled, $B$-coloured finite graph, and a patch is
said to \emph{match} a graph colouring $\tau\colon\Gamma\to B$ at a
vertex $v$ if there exists a graph morphism $\Delta\to\Gamma$ mapping
$\Delta$'s root to $v$ and preserving labels and colours. Then
Lemma~\ref{lem:pattern} says that the domino problem may be specified
by a collection of forbidden patches.

In this manner, the equivalent formulation of the tiling problem by
Wang tiles, for an undirected Schreier graph $X$, is as follows. An
instance of the problem is a finite set $B$ of colours and a set
$W\subseteq B^A$ of Wang tiles. A valid colouring is a map
$\tau\colon X\to W$ such that $\tau(x)(a)=\tau(a\cdot x)(a^{-1})$ for
all $x\in X,a\in A$.

It is sometimes interesting to consider the \emph{space} of tilings of
a graph given by a tileset: it is the space of graph morphisms
$\Gamma\to\Delta$, with its natural topology. Assuming that
$\Gamma=(V,E,\lambda)$ is properly labelled, for a tileset $\Theta$ we
define
\[X_\Theta\coloneqq\{\tau\in B^V:\forall e\in E: (\tau(e^-),\lambda(e),\tau(e^+))\in\Theta\}.
\]
It is a closed subspace of $B^V$ for the product topology. It is also
invariant under the automorphism group of $\Gamma$, so is a
``$\Aut(\Gamma)$-flow of finite type''. (If furthermore $\Aut(\Gamma)$
were simply transitive on $V$, it would be an $\Aut(\Gamma)$-subshift
of finite type.)

The following result says that certain local configurations (for
example closed paths) may be marked in labelled graph by means of a
tileset, if we accept that sometimes more vertices than desired will
be marked:
\begin{lem}\label{lem:localmark}
  For every finite, rooted, labelled graph $(\Delta,x_0)$, there exists
  a tileset $\Theta\subseteq B\times A\times B$ and a subset
  $C\subseteq B$, such that
  \begin{enumerate}
  \item every $B$-colouring $\tau$ matching $\Theta$ satisfies
    \[\tau^{-1}(C)\supseteq\{v\in V:\exists(\Delta,x_0)\to(\Gamma,v)\};\]
  \item there exists a $B$-colouring $\tau$ matching $\Theta$ and satisfying
    \[\tau^{-1}(C)=\{v\in V:\exists(\Delta,x_0)\to(\Gamma,v)\}.\]
  \end{enumerate}
\end{lem}
\begin{proof}
  It suffices to consider the case of $\Delta$ being a single,
  $a$-labelled loop; the general case follows from
  Lemma~\ref{lem:pattern}. Select then $B=\{0,1,2,3\}$ and $C=\{0\}$,
  and
  \[\Theta=\{(0,a,0),(i,a,j)\forall i\neq j\in\{1,2,3\},(i,x,j)\forall x\neq a\in A\forall i,j\in\{0,1,2,3\}\}.
  \]
  No condition is imposed on $x$-labelled edges for $x\neq a$; all
  $a$-labelled loops must be coloured $0$ while every $a$-labelled
  path (closed or not) may be labelled either entirely by $0$, or
  alternating in $1,2,3$ (we need three colours to cover all the cases
  of an even-length cycle, an odd-length cycle, or an open path).
\end{proof}

\subsection{Seeded domino problems}\label{ss:seeded}
We shall also consider a variant of the domino problem in which the
graph has a distinguished vertex, which has to be given a specific
colour. Here is a formulation in terms of a rooted Schreier graph
$(X,x_0)$: an instance of the \emph{seeded} domino problem is a
collection $\Theta\subseteq B\times A\times B$ of dominoes with a
chosen $b_0\in B$; the question is whether there exists an assignment
$\tau\colon X\to B$ with $\tau(x_0)=b_0$ and
$(\tau(x),a,\tau(a x))\in\Theta$ for all $x\in X,a\in A$. If the
seeded domino problem is solvable on a graph $\Gamma$, then so is the
domino problem (by querying the seeded domino problem with all
possible choices of colour at $x_0$). It could well be that each time
the seeded domino problem is unsolvable, so is the domino problem;
though for graphs such as the square grid $\Z^2$, or tessellations of
the hyperbolic plane, it took substantially more effort to prove the
latter than the former.

Let $\Gamma=(V,E,\lambda)$ be an $A$-labelled graph, and let $v\in V$
be a distinguished vertex. The \emph{sunny-side-up} is the subset
$S_v\subseteq\{0,1\}^V$ consisting of colourings $V\to\{0,1\}$ with a
single `$1$' at an arbitrary position in the $\Aut(\Gamma$)-orbit of
$v$; and additionally, if the $\Aut(\Gamma)$-orbit of $v$ is infinite,
the all-$0$ configuration. In particular, if $\Gamma$ is infinite and
vertex-transitive, for instance a Cayley graph, then $S_v$ is
naturally in bijection with the one-point compactification
$V\cup\{\infty\}$ of $V$. We call $S_v$ \emph{sofic} if there exists a
tileset $\Theta\subset B\times A\times B$ and a map
$\pi\colon B\to\{0,1\}$ such that $S_v=\pi\circ X_\Theta$, namely
$S_v$ is obtained by projecting all valid $\Theta$-tilings through
$\pi$.
\begin{lem}\label{lem:ssu}
  Let $\Gamma$ be an $A$-labelled graph, with $v$ a vertex as above,
  and assume that $v$ has a finite $\Aut(\Gamma)$-orbit.  If the
  sunny-side-up $S_v$ is sofic then the seeded and unseeded domino
  problems are reducible to each other.
\end{lem}
\begin{proof}
  The unseeded tiling problem can always be solved by querying
  finitely many times the seeded tiling problem, with all choices of
  colours. Conversely, given an instance of the seeded tiling problem
  $\Theta_0\subset B_0\times A\times B_0$ and distinguished colour
  $b\in B$, let $\Theta_1\subset B_1\times A\times B_1$ and
  $\pi\colon B_1\to\{0,1\}$ be an encoding of $S_v$, and consider the
  tileset
  \[\Theta\coloneqq\{((b_0,b_1),a,(b_0',b_1')):(b_0,a,b_0')\in\Theta_0,(b_1,a,b_1')\in\Theta_0,\pi(b_1)=1\Rightarrow b_0=b\}.\]
  Any valid tiling by $\Theta$ consists of a valid tiling of
  $\Theta_0$ which furthermore has $b$ at a position marked by $S_v$,
  so $\Theta$ tiles if and only if $\Theta_0$ tiles with colour $b$ at
  $v$.
\end{proof}

%%%%%%%%%%%%%%%%%%%%%%%%%%%%%%%%%%%%%%%%%%%%%%%%%%%%%%%%%%%%%%%% 
\section{Self-similar graphs and spaces}\label{ss:ss}
Let us first recall how graphs appear in connection with self-similar
fractals. Following Kigami~\cite{kigami:harmonic}, consider a
compact set $K$ with a collection $\{F_s:s\in S\}$ of injective
continuous self-maps, and assume that there is a surjective continuous
map $\pi\colon S^{-\N}\to K$ with $\pi(w s)=F_s(\pi(w))$ for all
$s\in S,w\in S^{-\N}$. (The reason we write sequences as left-infinite
will soon become clear.) The map $\pi$, if it exists, is unique, and
the data $(K,S)$ are called a \emph{self-similar structure}. We denote
by $\sigma$ the shift map on $S^{-\N}$, defined by $\sigma(w s)=w$.

For a word $v\in S^*$ we set $K_v=\pi(S^{-\N}v)$; these are \emph{tiles}
covering $K$. The \emph{critical set} $C\subseteq S^{-\N}$ is
\[C=\pi^{-1}\bigg(\bigcup_{s\neq t\in S}K_s\cap K_t\bigg),\]
and the \emph{post-critical set} is
\[P=\bigcup_{n\ge1}\sigma^n(C).
\]
A self-similar structure is \emph{post-critically finite} if $P$ is
finite. A simple example that is worth keeping track of is the
following: $K=[0,1]$ and $S=\{0,1\}$ with $F_i(x)=(x+i)/2$. Then
$C=\{{}^\infty01,{}^\infty10\}$ and $P=\{{}^\infty0,{}^\infty1\}$,
with tiles the intervals
$K_{w_1\dots w_n}=[w_1/2+\cdots+w_n/2^n,w_1/2+\cdots+w_n/2^n+1/2^n]$.

Self-similar structures naturally yield graphs as follows: for
$n\in\N$, consider the graph $\Gamma_n$ with vertex set $S^n$, and an
edge between $v$ and $w$ whenever $K_v\cap
K_w\neq\emptyset$. Moreover, it is possible to consider ``limits'' as
$n\to\infty$ of these graphs, by ``zooming'' for all $n\in\N$ at the
basepoints $\xi_1\dots\xi_n$ given as prefixes of an infinite word
$\xi\in S^\N$. (Note here that $\xi$ is a right-infinite word!). This
can be seen more formally as follows: for $\xi\in S^\N$, consider the
ascending union
$\widehat K(\xi)=(K\times\N)/((x,n)=(F_{\xi_n}(x),n+1)\forall
n\in\N)$. It is naturally tiled by the tiles of the form
$K_v\times\{|v|\}$ for all words $v\in S^*$, and we may again form a
graph $\Gamma(\xi)$ with vertex set the collection of all tiles, if
one remembers that $K_v\times\{n\}$ and $K_{v\xi_{n+1}}\times\{n+1\}$ are
identified for all $v\in S^n$. Still in our example,
$\widehat K(\xi)=\R$, unless $\xi$ eventually ends in $0^\infty$ when
$\widehat K(\xi)=\R_+$, or $\xi$ eventually ends in $1^\infty$ when
$\widehat K(\xi)=\R_-$. The corresponding graphs are respectively
$\Z$, $\N$ and $-\N$.

In case $(K,S)$ is post-critically finite, Kigami shows
in~\cite{kigami:harmonic}*{Appendix~A} that it may be reconstructed
from a small amount of combinatorial data: an \emph{ancestor
  structure} is $(V,U,\{G_s\})$ for two finite sets $V\subseteq U$ and
a collection of injective maps $G_s\colon U\to V$, such that
$V = \bigcup_{s\in S} G_s(U)$, and if $U\neq\emptyset$ then
$G_s(U)\setminus U\neq\emptyset$ for all $s\in S$. Starting from a
post-critically finite self-similar structure, $U=\pi(P)$ and
$V=\pi(P S)$ and $G_s=F_s\restriction U$ define an ancestor structure.
Conversely, an ancestor structure determines a compact set $K$ and a
self-similar structure as follows: for $x\in V$, define
\[A_x = \{w=(w_n)\in S^{-\N}: \exists(x_n)\in U^{-\N}\text{ with }G_{w_{-1}}(x_{-1})=x\text{ and }w_n(x_n)=x_{n+1}\forall n\le-2\}.
\]
Set then
\[K=S^{-\N}/(v u\sim w u\text{ if $v=w$ or }\exists x\in V: v,w\in A_x),
\]
with for all $s\in S$ a map $F_s\colon K\to K$ induced by
$w\mapsto w s$ on $S^{-\N}$. It is easy to check that this
construction recovers the original $(K,S)$. Again in our example,
$U=\{0,1\}$ and $V=\{0,\tfrac12,1\}$ with
$A_{1/2}=\{{}^\infty01,{}^\infty10\}$.

\subsection{Bounded transducers}\label{ss:bounded}
An algebraic formalism described by Bondarenko and Nekrashevych
in~\cite{bondarenko-n:pcf} is closely related to the ancestor
structures above.

We recall that a \emph{self-similar group} is a group $G$ endowed with
a map $\Phi\colon G\times S\to S\times G$ for some finite set
$S$, satisfying for all $g,h\in G$ and $s\in S$ the condition
\[\Phi(g h,s)=(s'',g' h')\text{ whenever }\Phi(h,s)=(s',h')\text{ and }\Phi(g,s')=(s'',g').\]
This is equivalent to requiring that $S\times G$ admits the structure
of a \emph{$G$-$G$-biset}: it has two commuting $G$-actions, given by
$g\cdot(s,h)\cdot k=(s',g' h k)$ if $\Phi(g,s)=(s',g')$.

From the self-similarity map $\Phi$ one constructs an action of $G$ on
$S^\N$ as follows: given a word $\xi=\xi_1 \xi_2\dots\in S^\N$ and an
element $g\in G$, to define $g(\xi)$ set $g_0=g$ and for every $n\ge1$
set $(\xi'_n,g_n)\coloneqq\Phi(g_{n-1},\xi_n)$; then $g(\xi)=\xi'_1 \xi'_2\dots$. In
other words, we have a recursive formula
$g(\xi_1 \xi_2\dots)=\xi'_1\,g_1(\xi_2\dots)$ with
$\Phi(g,\xi_1)=(\xi'_1,g_1)$. The same formulas may be used to define an
action of $G$ on the set $S^n$ of words of length $n$. In fact, the
map $\Phi$ may be extended to a map
$\Phi\colon G\times S^*\to S^*\times G$ by $\Phi(g,u v)=(u' v',g'')$
whenever $\Phi(g,u)=(u',g')$ and $\Phi(g',v)=(v',g'')$, and
$\Phi(g,\varepsilon)=(\varepsilon,g)$ for $\varepsilon\in S^*$ the
empty word; then the action $g(v)$ is the first co\"ordinate of
$\Phi(g,v)$.

A fundamental example is afforded by the infinite cyclic group
$G=\langle t\rangle$ and $S=\{0,1\}$, with
\[\Phi(t^{2n},0)=(0,t^n),\quad\Phi(t^{2n},1)=(1,t^n),\quad\Phi(t^{2n+1},0)=(t^n,1),\quad\Phi(t^{2n+1},1)=(t^{n+1},0).\]
We shall return regularly to this example, called the
\emph{Kakutani-von Neumann odometer}.

The self-similarity map $\Phi$ may conveniently be viewed as a graph,
with vertex set $G$ and an edge from $g\in G$ to $h\in G$ labelled
`$s|t$' whenever $\Phi(g,s)=(t,h)$. Then the action of $G$ on $S^\N$
is understood as follows: for $g\in G$ and $\xi\in S^\N$, find the
unique right-infinite path in the graph that starts at $g$ and has $\xi$
as the left components of its labels. Then $g(\xi)$ is the word read on
the right components of the labels along that same path. This graph is
called the \emph{full transducer} of the self-similar group.

A self-similar group is called \emph{recurrent} if the map
$\Phi\colon G\times S\to S\times G$ is onto. This implies, in
particular, that the action of $G$ is transitive on $S^n$ for all
$n\in\N$.  A self-similar group is \emph{contracting} if there exists
a finite subtransducer $N$ of the full transducer such that every path
is eventually contained in $N$. We indifferently use $N$ for the
subtransducer or the corresponding subset of $G$. The minimal such $N$
is called the \emph{nucleus} of the action. Following Bondarenko and
Nekrashevych~\cite{bondarenko-n:pcf}*{Definition~5.1}, a contracting
self-similar group is \emph{post-critically finite}, a.k.a.\
\emph{bounded}, if its nucleus contains only a finite number of
left-infinite paths ending at a non-trivial state. The
\emph{post-critical set} of $G$ is the set $P$ of left-infinite words
read as inputs along these paths. In the example of the odometer, the
nucleus is
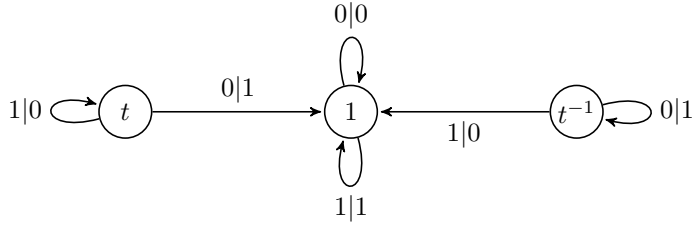
\[\begin{fsa}
    \node[state,minimum size=7mm] (t) at (-3,0) {$t$};
    \node[state,minimum size=7mm,inner sep=0mm] (T) at (3,0) {$t^{-1}$};
    \node[state,minimum size=7mm] (1) at (0,0) {$1$};
    \draw (t) edge[loop left] node {$1|0$} ()
    (t) edge node {$0|1$} (1)
    (1) edge[loop above] node {$0|0$} ()
    (1) edge[loop below] node {$1|1$} ()
    (T) edge[loop right] node {$0|1$} ()
    (T) edge node {$1|0$} (1);
  \end{fsa}\]
and the post-critical set is $P=\{{}^\infty0,{}^\infty1\}$.

\begin{prop}[\cite{nekrashevych:ssg}*{Proposition~2.11.3}]
  Let $G$ be a self-similar, contracting, finitely generated,
  recurrent group. Then $G$ is generated by its nucleus.\qed
\end{prop}

Let $G$ be a contracting self-similar group with alphabet $S$ and
nucleus $N$, and recall the following fundamental construction by
Nekrashevych: define
\[L(G)=S^{-\N}/(w\sim w'\Longleftrightarrow \exists (g_n)\in N^{-\N}\text{ with }\Phi(g_n,w_n)=(w'_n,g_{n+1})\forall n<0).
\]
Then $L(G)$ is a topological space called $G$'s \emph{limit space},
and the dynamical system on $L(G)$ induced by the shift map on
$S^{-\N}$ lies in a duality relation with $G$. Note that $L(G)$ in
fact admits an orbispace structure, with finite isotropy groups, and
that the duality between self-similar groups and expanding dynamical
systems holds only when this orbispace structure is taken into
account. We choose to ignore it here, and again refer
to~\cite{nekrashevych:ssg} for details and extra information; in
particular, the limit space $L(G)$ is compact, metrizable, has finite
topological dimension, and is connected as soon as $G$ is
recurrent. Define next
\[T(G)=S^{-\N}/(w\sim w'\Longleftrightarrow \exists (g_n)\in N^{-\N}\text{ with }\Phi(g_n,w_n)=(w'_n,g_{n+1})\forall n<0\text{ and }g_0=1).
\]
Following~\cite{nekrashevych:ssg}*{\S3.3}, we call $T(G)$ the
\emph{digit tile}, noting that $L(G)$ is naturally a quotient of
$T(G)$. Note that $T(G)$ is a topological space, i.e.\ does not have
singular orbispace points. There are natural maps
$F_s\colon T(G)\to T(G)$ induced by the maps $w\mapsto w s$ on
$S^{-\N}$. More generally, for a word $v\in S^*$ we let $T_v(G)$ be
the image of $S^{-\N}v$ in $T(G)$. Still in the example of the
odometer, the limit space $L(G)=[0,1]/(0\sim1)$ is the circle, with
expanding self-covering induced by $f(x)=2x\bmod 1$; and the digit
tile $T(G)=[0,1]$ is the interval, with contractions $F_i(x)=(x+i)/2$.

We note the following connection between the definitions of Kigami and
Nekrashevych. Let $G$ be a post-critically finite self-similar group, with post-critical set $P\subset S^{-\N}$. Set $U=P$ and
\[V=(P S)/(u\sim v\Longleftrightarrow\text{ there is a path in the
    nucleus $N$, ending at $1$, with labels $u|v$});\] note that we
have $U\subseteq V$ and maps $G_s\colon U\to V$ given by
$p\mapsto [p s]_\sim$ for all $s\in S$.
\begin{prop}
  If $G$ is a post-critically finite self-similar group, then
  $(V,U,\{G_s\}_{s\in S})$ is an ancestor structure, and the
  corresponding self-similar structure $(K,S)$ is homeomorphic to the
  digit tile of $G$.
\end{prop}
\begin{proof}
  It is easy to check the axioms $V = \bigcup_{s\in S} G_s(U)$ and
  $G_s(U)\setminus U\neq\emptyset$ for all $s\in S$ if
  $U\neq\emptyset$. It then suffices to note that the construction of
  $(K,S)$ from an ancestor structure coincides with the construction
  of $T(G)$. Following the definitions, the non-trivial equivalence
  classes $A_x$ are, for $x\in V$,
  \begin{align*}
    A_x &= \{(w_n)\in S^{-\N}:\exists (x_n)\in U^{-\N}: G_{w_n}(x_n)=x_{n+1}\forall n\le-2\text{ and }G_{w_{-1}}(x_{-1})=x\}\\
    &= \{w\in S^{-\N}:[w]_\sim=x\}.
  \end{align*}
  Thus the equivalence relation constructing $K$ from $(V,U,S)$
  identifies two left-infinite words $w,w'$ precisely when there exists
  a left-infinite path in the nucleus $N$ with label $w|w'$ and ending
  at $1$, and this is the equivalence relation constructing the digit
  tile.
\end{proof}

\begin{prop}
  Let $G$ be a critically finite self-similar group, and let $(K,S)$
  be its associated self-similar structure. Consider a ray
  $\xi\in S^\N$ and the Schreier graph of the orbit $G\cdot\xi$ with
  generating set the nucleus of $G$. Then the tile adjacency graph
  $\Gamma(w)$ is the simple graph associated with the Schreier graph:
  the graph obtained by removing all loops and combining multiple
  edges.
\end{prop}
\begin{proof}
  Since both the Schreier graph and the tile adjacency graph are
  inductive limits, it suffices to check the statement for the
  following two graphs: the tile adjacency graph $\Gamma_n$ describing
  intersections of tiles $K_v$ with $v\in S^n$, and the graph obtained
  from the Schreier graph of $G$'s action on $S^n$, in which loops are
  removed, multiple edges are combined, and edges labelled $g\in N$
  from $u\in S^n$ to $v\in S^n$ are removed if $\Phi(g,u)=(v,h)$ with
  $h\neq1$.

  Now in the latter graph the remaining edges are edges labelled
  $g\in N$ from $u\in S^n$ to $v\in S^n$ with $\Phi(g,u)=(v,1)$, so
  there is a path in the nucleus $N$ starting at $g$ and ending at $1$
  with label $u|v$; so $u,v$ are respectively of the form $u' w$ and
  $v' w$ with $u',v'$ suffixes of critical left-infinite
  words. Therefore the tiles $K_u$ and $K_v$ intersect.

  Conversely, if $u,v\in S^n$ are such that the tiles $K_u,K_v$
  intersect, then $u=u' w$ and $v=v' w$ for some word $w$ and suffixes
  $u',v'$ of critical left-infinite words. There is then a
  left-infinite path in the nucleus whose label ends in $u'|v'$, so
  there exists $g\in N$ with $\Phi(g,u')=(v',1)$; thus $g\cdot u=v$
  and there is an edge in the Schreier graph from $u$ to $v$.
\end{proof}

%%%%%%%%%%%%%%%%%%%%%%%%%%%%%%%%%%%%%%%%%%%%%%%%%%%%%%%%%%%%%%%% 
\section{Decidability results}
The main result of this section is that Schreier graphs of
post-critically finite self-similar groups have decidable monadic
second-order theory. Note that there is one Schreier graph per ray
$\xi\in S^\N$, and all these graphs are non-isomorphic as rooted
graphs --- so there are continuously many different graphs. However,
they all have the same collection of balls, and therefore the same
answers to a given tiling problem, except in case $\xi$ is ultimately
periodic with same period as a post-critical ray. The situation with
respect to monadic second-order theory is a bit less clear.

We begin by the domino problem, for which a direct argument is
possible:
\begin{prop}\label{prop:pcfdecidable}
  Let $G$ be a post-critically finite self-similar group acting on
  $S^\N$, let $\xi\in S^\N$ be a ray, and let $\Gamma$ be the Schreier
  graph of the orbit $G\xi$ with respect to the nucleus of $G$.

  Then the domino problem on $\Gamma$ is decidable.
\end{prop}
\begin{proof}
  We shall give an algorithm that decides the domino problem. To fix
  notation, let $N$ denote the nucleus of $G$, and let $P\subset S^{-\N}$
  denote the post-critical set of $G$.
  
  Let $\Theta\subseteq B\times N\times B$ be an instance of the domino
  problem for $\Gamma$.

  Consider for all $n\in\N$ the ``tile Schreier graph'' $\Gamma_n$
  obtained from $G$'s action on $S^n$: the vertex set is $S^n$, and
  there is an edge from $v$ to $w$ labelled $g\in N$ whenever
  $\Phi(g,v)=(w,1)$. It is a subgraph of the usual Schreier graph of
  $G$'s action on $S^n$. Let $P_n$ denote the collection of length-$n$
  suffixes of post-critical words. Note that the vertices in
  $\Gamma_n$ that have fewer than $\#N$ neighbours are precisely those
  in $P_n$.
  
  Let $\Lambda_n$ denote the set of restrictions to $P_n$ of valid
  colourings of $\Gamma_n$; more precisely, $\Lambda_n$ is the subset
  of $B^{P_n}$ consisting of all $\lambda\colon P_n\to B$ such that
  the colouring via $\lambda$ of $P_n$ can be extended to a valid
  colouring of $\Gamma_n$. Clearly $P_0=\{\varepsilon\}$ and
  $\Lambda_0=B^{P_0}$.

  Let us consider now how to compute $\Lambda_{n+1}$ from
  $\Lambda_n$. Start with a collection of $\#S$ colourings
  $(\lambda_s)_{s\in S}$ of $\Gamma_n$, and use $\lambda_s$ to colour
  $P_n s\subset\Gamma_{n+1}$. Keep only those colourings that match on
  their inner edges: for all $p\in P_n$, consider all $g\in N$ with
  $\Phi(g,p)\notin S\times\{1\}$, and then for all $s\in S$ write
  $\Phi(g,p s)=(q t,h)$; if $h=1$ then require
  $(\lambda_s(p),g,\lambda_t(q))\in\Theta$, while if $h\neq1$ then
  $p s\in P_{n+1}$, and set $\lambda(p s)=\lambda_s(p)$. We have in
  this manner defined a function $\lambda\colon P_{n+1}\to B$. Let
  $\Lambda_{n+1}$ be the collection of all the functions $\lambda$
  that can be obtained in this manner.

  Note that all post-critical points $p=\cdots p_{-2}p_{-1}\in P$ are
  pre-periodic. Therefore, for $n$ large enough, we have $\#P_n=\#P$
  and there is a canonical bijection between $P_n$ and
  $P$. Furthermore, for all $g\in N$ we have
  $\Phi(g,p_{-n}\dots p_{-1})=(q_{-n}\dots q_{-1},h)$ for a
  post-critical point $\cdots q_{-2}q_{-1}$, and $h$ depends only on
  the value of $n$ modulo $\ell$ for some least common period $\ell$
  of all post-critical points. Therefore, for $n$ large enough, the
  spaces of maps $B^{P_n}$ and $B^{P_{n+\ell}}$ are canonically in
  bijection, and the map
  $\Lambda_n\mapsto\Lambda_{n+1}\mapsto\cdots\mapsto\Lambda_{n+\ell}$
  only depends on $n\bmod\ell$; the ``$n$ large enough'', modulus
  $\ell$ and map $\tau\colon\Lambda_{n\ell}\mapsto\Lambda_{(n+1)\ell}$
  may all be computed from the nucleus. Since $B$ and $N$ are finite,
  $\tau$ is an ultimately periodic map on the family of subsets of
  $B\times N\times B$.

  There are now two possibilities. Either $\tau$ eventually reaches
  the empty set, in which case there is no valid tiling of $\Gamma$;
  or $\tau$ ultimately cycles along non-empty sets, in which case
  there exist valid tilings of $\Gamma_n$ for all $n$.
  
  This last case subdivides in two. Either the ray $\xi$ is regular,
  and then $\Gamma$ is an ascending union of copies of $\Gamma_n$ so
  is tileable; or the ray $\xi$ is singular, namely is ultimately
  periodic with same period $p$ as a post-critical point
  ${}^\infty p\in P$. Then without loss of generality $\xi=p^\infty$,
  and the tiling of $\Gamma$ is obtained from a limit of tilings of
  $\Gamma_n$ by checking that, for $n$ large enough that the cycle of
  $\tau$ is attained, at least one colouring $\lambda\in\Lambda_n$ is
  valid at $\xi$; assuming without loss of generality that $n$ is
  divisible by $|p|$, this will hold precisely when
  $(\lambda(p^{n/|p|}),g,\lambda(g(p^{n/|p|})))\in\Theta$ for all
  $g\in N$ with cycle $p$, namely with $\Phi(g,p)=(q,g)$. Again this
  requires a finite amount of checking.
\end{proof}

Consider first the elementary spaces of tilings $X_\Theta$ and
$X_{\Theta,b_0}$ on a rooted graph $(\Gamma,x_0)$, the latter being
the solution set of a seeded domino problem with colour $b_0$ at the
root $x_0$. The \emph{$\Gamma$-regular languages} are those spaces of
colourings of $\Gamma$'s vertices obtainable from elementary ones by
boolean operations (intersection, complement) and projections from a
set of colours to another; and $\Mon(\Gamma)$ is decidable if and only
if the emptiness problem for $\Gamma$-regular languages is
decidable~\cite{muller-schupp:mso}*{Theorem~3.1}. It might be possible
to prove decidability of $\Mon(\Gamma)$ by extending the proof above
to boolean expressions and projections; but this seems difficult. We shall prove, by an entirely different method,
\begin{thm}\label{thm:msodecidable}
  Let $G$ be a post-critically finite self-similar group acting on
  $S^\N$, let $\xi\in S^\N$ be an eventually periodic ray, and let
  $\Gamma$ be the Schreier graph of the orbit $G\xi$ with respect to
  the nucleus of $G$.

  Then the monadic second-order theory of the rooted graph
  $(\Gamma,\xi)$ is decidable.
\end{thm}
In fact, we shall prove something stronger, namely $G$ need only be a
post-critically finite inverse semigroup of partially-defined
bijections of $S^\N$.

\begin{proof}
  We begin by a straightforward reduction: up to replacing $S$ by a
  power of itself, and a small modification, we may assume that $\xi$
  is a constant ray $(s_0)^\infty$. By the same replacement, we may
  also assume that the generators of $G$ have a very specific form,
  which we will explicit later.

  The first step is to encode the vertices of $\Gamma$. We shall view
  them as leaves of a tree, whose level-$n$ vertices are
  $\sigma^n(G\xi)$. More precisely, the vertex set of the tree $T$ is
  $\bigsqcup_{n\ge0}\sigma^n(G\xi)\times\{n\}$, and there is an edge
  labelled $s$ from $(s\eta,n)$ to $(\eta,n+1)$ for all
  $s\eta\in\sigma^n(G\xi)$. Thus $T$ is an $S$-labelled tree, rooted at
  $s_0^\infty$.

  I claim that $\Mon(T)$ is decidable. Indeed consider first the tree
  $T_0$ of prefixes of the set of words
  $\{a^n b c^{n-1}:n\ge1\}\cup\{\varepsilon\}$. This is the set of
  total states of a push-down automaton, so $\Mon(T_0)$ is decidable
  by the main result of~\cite{muller-schupp:mso}. We then apply the
  finite, regular mapping
  $a\mapsto s_0, b\mapsto (S\setminus\{s_0\})^{-1},c\mapsto S^{-1}$ to
  obtain an $S$-labelled graph (with some of the labels written in
  reverse as $s^{-1}$) isomorphic to $T$. Since regular mappings
  preserve decidability, the claim is proven.

  We next show how the Schreier graph $\Gamma$ may be interpreted in
  $T$. First, it will be convenient to define successors, independently
  of the $S$-labelling:
  \[\text{succ}(v,w)\equiv\bigvee_{s\in S}s\cdot v=w.
  \]
  Vertices of $\Gamma$ are leaves of $T$, and are thus characterized
  by the formula
  \[\text{leaf}(v)\equiv\neg\exists w[\text{succ}(w,v)].
  \]

  Given a monadic predicate $\phi$ satisfied by a set $X$, one may
  consider a predicate characterizing the minimal such $X$, namely
  \[\phi^{\min}(X)\equiv\phi(X)\wedge\forall Y[\phi(Y)\Rightarrow X\subseteq Y].
  \]
  If $\phi$ takes the form `$\exists X\dots$', we simply write
  $\phi^{\min}$ as `$\exists^{\min} X\dots$'.

  Every leaf (or even vertex) of $T$ has a unique ray, obtained by
  following edges in $S$: the ray $R$ of $v$ is characterized by
  \[\text{ray}(v,R)\equiv \big(v{\in}R\wedge(\forall w{\in}R[\exists x{\in}R[\text{succ}(w,x)]])\big)^{\min},
  \]
  namely it is a minimal set containing $v$ and in which every element
  has a successor. In particular, the order in the tree may be defined
  by
  \[(v\preceq w)\equiv\forall Q,R[\text{ray}(v,Q)\wedge\text{ray}(w,R)\wedge R{\subseteq}Q].
  \]
  The first symbol of the word coding a vertex $v$ is given by
  predicates
  \[\text{head}_s(v)\equiv\exists w[s\cdot v=w]\text{ for all }s\in S.\]
  Finally we consider predicates that a ray, or portion of ray, has a given constant label:
  \[\text{const}_q(R)\equiv\forall x{\in}R[\text{head}_q(x)]\text{ for all }q\in S.\]
  
  Up to replacing $S$ by a power of itself, as we did in the very
  first step, we may assume that every generator of $G$ is a disjoint
  union (as a relation) of elementary partial bijections, obtained as
  follows:
  \begin{enumerate}
  \item The identity is elementary.
  \item If $g$ is elementary and $s,t\in S$, then the partial bijection $h$ defined only by $\Phi(h,s)=(t,g)$ is elementary.
  \item If $g$ is elementary and $q,r,s,t\in S$ with $q\neq s$ and
    $r\neq t$, then the partial bijection $h$ defined by
    $\Phi(h,s)=(t,g)$ and $\Phi(h,q)=(r,h)$ is elementary.
  \end{enumerate}
  Indeed the transducer of a bounded transformation of $S^\N$ consists
  of a finite collection of cycles, reached by finite paths, and
  leading to the identity. The cycles can be assumed to be loops, and
  each transition to the identity is considered separately.

  We prove, by induction, that every elementary partial bijection may
  be encoded by a monadic second-order formula. This is obvious for
  the identity transformation, which is coded as
  \[\text{id}(v,w)\equiv (v=w).
  \]
  Assume that $g$ is elementary so there is a formula for $g$, and
  consider $h$ with $\Phi(h,s)=(t,g)$. Then a formula for $h$ is
  \[h(v,w)\equiv\text{head}_s(v)\wedge\text{head}_t(w)\wedge\text{succ}(v,v')\wedge\text{succ}(w,w')\wedge g(v',w').
  \] All these operations only required first-order
  logic. For the last case, assume that there is a formula for $g$,
  and $h$ is defined by $\Phi(h,s)=(t,g)$ and $\Phi(h,q)=(r,h)$. Then a formula for $h$ is
  \begin{multline*}
    h(v,w)\equiv\exists P,Q,R\Big[\text{ray}(v,Q)\wedge\text{ray}(w,R)\wedge P{=}Q{\cap}R\wedge{}\\
    \big((P{=}\emptyset\wedge\text{const}_q(Q)\wedge\text{const}_r(R))\vee
    (\exists x{\in}P[\neg\exists y{\in}P[\text{succ}(y,x)]\wedge{}\\
    \exists v'{\in}Q,w'{\in}R[{s\cdot v'{=}x=t\cdot w'}\wedge{\text{const}_q(Q\setminus P\setminus\{v'\})}\wedge{\text{const}_r(R\setminus P\setminus\{w'\})}]])\big)\Big].
  \end{multline*}
  The meaning is the following: $P$ is the common suffix of the words
  encoding $v,w$. If these words have no common suffix, then they must
  respectively be $q^\infty$ and $r^\infty$. Otherwise, they must
  respectively have the form $q^* s$ and $r^* t$.

  Since we can interpret the action of the generators of $G$ in
  $\Mon(T)$, which is decidable, it follows that $\Mon(\Gamma)$ is
  decidable as well.
\end{proof}

\subsection{Treewidth}\label{ss:treewidth}
Let $\Gamma$ be a graph. A \emph{tree decomposition} of $\Gamma$ is a
tree $\Delta$ with, for each vertex $v\in V(\Delta)$, a subset
$X_v\subseteq V(\Gamma)$ called a \emph{bag}, subject to two axioms:
\begin{enumerate}
\item For every $v\in V(\Gamma)$ the set of bags containing $v$ spans
  a non-empty subtree of $\Delta$;
\item For every $e\in E(\Gamma)$ there is a bag containing
  $\{e^+,e^-\}$.
\end{enumerate}
The \emph{width} of a tree decomposition is one less than the supremum
of the bag sizes, and the \emph{treewidth} of $\Gamma$ is the minimal
width of a tree decomposition. (The ``one less'' is purely aesthetic,
and implies that trees have treewidth $1$.)

Decidability of the monadic second-order theory of graph implies that
its treewidth is bounded~\cite{seese:mso}; but this can be proven
directly:
\begin{prop}
  Let $\Gamma$ be the Schreier graph of a post-critically finite
  group. Then the treewidth of $\Gamma$ is finite.
\end{prop}
We shall actually give a computable bound on the treewidth: let $G\xi$
be the vertex set of $\Gamma$, let $P$ be $G$'s post-critical set, let
$S$ be its alphabet, and in the transducer defining $G$ let $p,q$ be
respectively the maximal length of a path leading from a generator to
a cycle and from a cycle to the identity. Then the treewidth is at most
$\#P\cdot\#S^{p+q}$.
\begin{proof}
  As in the proof of Theorem~\ref{thm:msodecidable}, let $\Delta$ be
  the following tree: its vertex set is
  $\bigsqcup_{n\ge0}\sigma^n(G\xi)\times\{n\}$, and there is an edge
  between $(\eta,n)$ and $\sigma(\eta),n+1)$.  The bag at $(\eta,n)$
  is
  \[X_{(\eta,n)}=\{u v w\sigma^{p+q}(\eta):\text{$|u|=p$ and $v$ is the
      length-$n$ suffix of a post-critical ray $\in P$ and $|w|=q$}\}\cap G\xi.
  \]
  The cardinality estimate is obviously satisfied, and the bags
  containing any given $\eta\in G\xi$ form an $\#S$-regular rooted
  subtree of $\Delta$. Furthermore every edge of $\Gamma$, say with
  label $g\in G$, connects $u v w\eta$ to $u'v'w'\eta$ where $v$ is
  the suffix of a post-critical ray: $\Phi(g,u)=(u',h)$ where $h$ lies
  on a cycle or is already finitary, $\Phi(h,v)=(v',k)$ with $k$
  finitary, and $\Phi(k,w)=(w',1)$.
\end{proof}

\subsection{The Sierpiński gasket}\label{ss:sierpinski}
The Sierpiński gasket is ubiquitous in discussions on fractals and
graphs, and this text shall not be an exception. A remarkably simple
transducer produces the gasket as its digit tile, and the associated
graphs as Schreier graphs, see~\cite{grigorchuk-s:hanoi}:\\
\pgfdeclarelindenmayersystem{BW Sierpinski triangle}{
  \symbol{X}{\pgflsystemdrawforward}
  \symbol{Y}{\pgflsystemdrawforward}
  \rule{X -> X-Y+X+Y-X}
  \rule{Y -> YY}
}
\pgfdeclarelindenmayersystem{Sierpinski triangle}{
  \symbol{X}{\draw[red,sierp,-] (0,0) -- (\pgflsystemcurrentstep,0); \pgflsystemmoveforward}
  \symbol{Y}{\draw[green,sierp,-] (0,0) -- (\pgflsystemcurrentstep,0); \pgflsystemmoveforward}
  \symbol{Z}{\draw[blue,sierp,-] (0,0) -- (\pgflsystemcurrentstep,0); \pgflsystemmoveforward}
  \rule{W -> WfW}
  \rule{T -> +UWZ-W---U++YU-W-XW+++}
  \rule{U -> -TWZ+W+++T--YT+W+XW---}
}
\centerline{\begin{fsa}[baseline,scale=0.8]
    \node[state,red] (a) at (0:2) {$a$};
    \node[state,green] (b) at (120:2) {$b$};
    \node[state,blue] (c) at (240:2) {$c$};
    \node[state] (e) at (0,0) {$e$};
    \path (a) edge[out=-30,in=30,looseness=8] node[above=1mm] {$2|2$} (a)
    (b) edge[out=90,in=150,looseness=8] node[left] {$1|1$} (b)
    (c) edge[out=210,in=270,looseness=8] node[left] {$0|0$} (c)
    (a) edge node {$0|1,1|0$} (e)
    (b) edge node[left] {$0|2,2|0$} (e)
    (c) edge node[left] {$1|2,2|1$} (e);
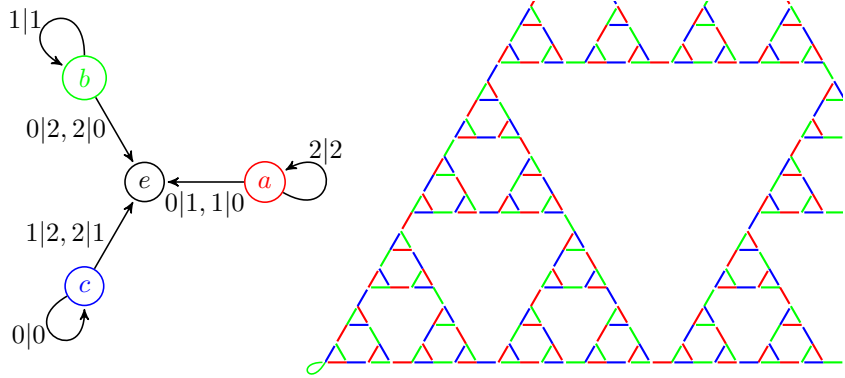
    \begin{scope}[xshift=3cm,yshift=-3cm,scale=2,sierp/.style={thick}]
      \draw[green,-] (0,0) edge [in=180,out=240,min distance=2.5mm] (0,0);
      \clip (0,-0.1) rectangle (4.3,3);
      \draw[-] (0,0) l-system [l-system={Sierpinski triangle, axiom=T, step=0.18cm, order=5, angle=60},fill=white];
    \end{scope}
  \end{fsa}}

The group $\langle a,b,c\rangle$ is known as the ``$3$-peg Hanoi tower
group'', since its Schreier graph is well-known to be related to
solutions to the Hanoi towers puzzle: generators $a,b,c$ correspond to
moving the top disk respectively between pegs $0$ and $1$, $0$ and
$2$, or $1$ and $2$.

%%%%%%%%%%%%%%%%%%%%%%%%%%%%%%%%%%%%%%%%%%%%%%%%%%%%%%%%%%%%%%%% 
\section{Undecidability results}\label{ss:undecidable}
We now prove that the domino problem is undecidable on some examples
of self-similar graphs. The method of proof is uniform: simulate a
grid within the graph; then, since machines may be simulated on
grids, we see that the graph in question is capable of universal
computation, and therefore has undecidable domino problem.

Let $\Gamma$ be an $A$-labelled graph. By ``simulating a grid in
$\Gamma$'' we mean the following: there is a grid $\Delta$, and a
domino problem for $\Gamma$ that marks some vertices in $\Gamma$ as
representing vertices of $\Delta$, and some sequences of edges in
$\Gamma$ as representing edges of $\Delta$.

Even more precisely, assume $\Delta$ is $B$-labelled. The colouring of
$\Gamma$ must distinguish some of its vertices as being
``$\Delta$-vertices'', and each edge of $\Gamma$ may carry some number
of signals in $\{0,1\}\times B\times\{0,1\}$. It is required that
$\Delta$ coincides with the graph with vertex set the
$\Delta$-vertices, and an edge labelled $b$ for every path in $\Gamma$
coloured $(0,b,1)\cdots(1,b,1)\cdots(1,b,0)$ and joining
$\Delta$-vertices. For details
see~\cite{bartholdi-salo:ll}*{\S2}.

If $\Delta$ has unsolvable domino problem, then so does $\Gamma$:
indeed given any instance of a domino problem on $\Delta$, it may be
combined with the domino problem defining the simulation so as to
produce a domino problem for $\Gamma$; if that last problem were
solvable, so would be the original one.

If $\Delta$ is the subgraph of $\Gamma$ consisting of all
$A'$-labelled edges, for some subset $A'\subset A$, then $\Delta$ is
simulated by $\Gamma$. This was exploited for example
in~\cite{jeandel:translation} to prove that every group containing a
direct product of two infinite, finitely generated subgroups has
unsolvable domino problem.  However, the examples we consider here are
graphs that do not contain any grid as a subgraph.

Conversely, if $\Gamma$ simulates $\Delta$ then $\Delta$ is (up to
duplicating the vertices and edges of $\Gamma$ by a finite amount) a
minor of $\Gamma$. However, our definition of simulation requires some
amount of regularity (in the sense of regular languages) in the
extraction of the minor from $\Gamma$.

In more detail: we first reduce the \emph{seeded} domino problem on
$\Gamma$ to that of a grid, by constructing a tileset that, when
properly seeded, exhibits a grid as a minor of $\Gamma$. It seems
impossible to avoid the seed, since the grid will necessarily be quite
sparse in $\Gamma$, and in particular will ignore arbitrarily large
balls. We then show, using a sunny-side-up tileset, that the seed may
be distinguished by an (unseeded) tileset. The (unseeded) domino
problem is then equivalent to the seeded one, by Lemma~\ref{lem:ssu}.

The precise implementation of this plan depends on the graph, and I
will carry it out for two examples of self-similar graphs that seem at
the border of decidability/undecidability. The second one is
essentially equivalent to~\cite{barbieri-sablik:ssdomino}. In each
case, a simulation has to be defined \emph{ad hoc}, and I will not
attempt to make any general claim.

\subsection{The long range graph}\label{ss:longrange}
Consider the graph with vertex set $\Z$, and two kinds of edges: all
edges between $n$ to $n+1$, and a loop at $0$ and for all
$m\in\Z,s\in\N$ an edge between $2^s(2m-1)$ and $2^s(2m+1)$. It is
known as the ``long range graph'', a deterministic avatar of
long-range percolation~\cite{schulman:longrange1d}. It is one of the
simplest examples of ``$\omega$-periodic graphs'' considered
in~\cite{benjamini-hoffman:omega}, see also~\cite{bondarenko:growth}.\\
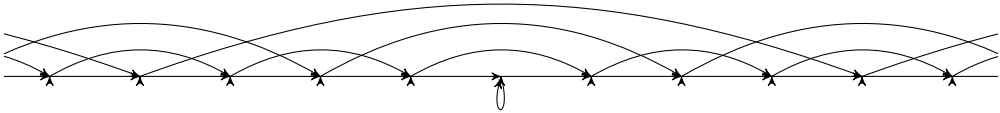
\centerline{\begin{tikzpicture}[scale=1.2,>=stealth']
    \clip (-5.5,-1) rectangle (5.5,1);
    \foreach\i in {-6,...,5} {
      \draw[->] (\i,0) -- +(1,0);
    }
    \foreach\i in {-7,-5,...,5} {
      \draw[->] (\i,0) edge[bend left] +(2,0);
    }
    \foreach\i in {-6,-2,...,5} {
      \draw[->] (\i,0) edge[bend left] +(4,0);
    }
    \foreach\i in {-12,-4,...,5} {
      \draw[->] (\i,0) edge[bend left=20] +(8,0);
    }
    \draw[->] (0,0) edge[loop below] ();
  \end{tikzpicture}}

This graph $\Gamma$ is the Schreier graph $G\cdot0^\infty$ of a
self-similar group given by the transducer\\
\centerline{\begin{fsa}[baseline,scale=1]
    \node[state] (t) at (0,0) {$t$};
    \node[state] (u) at (-3,0) {$u$};
    \node[state] (e) at (3,0) {$e$};
    \path (t) edge[loop above] node {$1|0$} ()
    (t) edge node {$0|1$} (e)
    (u) edge[loop left] node {$0|0$} ()
    (u) edge node {$1|1$} (t)
    (e) edge [loop right] node {$0|0,1|1$} ();
  \end{fsa}}

Indeed identify $\Z$ with all infinite words in $\{0,1\}^\infty$
ending in $0^\infty$ or $1^\infty$, via the binary expansion of
integers. Then the edges given by generator $t$ connect $n$ to $n+1$,
while generator $u$ connects $2^s(2m-1)$ to $2^s(2m+1)$. Set $A=\{t,u\}$.

\begin{prop}\label{prop:lrsimul}
  There is a tileset $\Theta_0\subset B_0\times\{t,u\}\times B_0$, and
  a colour $b_0\in B_0$, such that $X_{\Theta_0}$ contains a unique
  tiling $\tau$ with $\tau(0)=b_0$; and this tiling simulates a grid.
\end{prop}
\begin{proof}
  We use dominoes to impose successively more colourings on $\Gamma$;
  the colours combine, so that each vertex will have many different
  colours at the end of the process. We are in fact imposing a
  sequence of domino colourings on $\Gamma$, with each one making use
  of the previous colours.
  \begin{enumerate}
  \item We first mark $0$ with a special colour $0$ (this is specified
    by the seed colour $b_0$).
  \item We mark all positive integers by $+$, and all negative ones by
    $-$. This is done by choosing $B_1=\{0,+,-\}$ and
    $\Theta_1=\{(-,t,-),(-,t,0),(0,t,+),(+,t,+)\}\cup(B_1\times\{u\}\times
    B_1)$.
  \item We mark all powers of $2$ by $p$. This is done by selecting
    all vertices marked $+$ and whose $u^{-1}$-neighbour is marked
    $-$: choose $B_2=B_1\times\{p,\mathvisiblespace\}$ and
    \[\Theta_2=\{((b,c),a,(b',c')):((b,a,b')=(-,u,+))\Leftrightarrow(c'=p)\}.
    \]
  \item Mark all strips of integers between powers of $2$ as $e$
    (even) or $o$ (odd), starting with $-\N$ marked as $e$. This is
    done by setting $B_3=B_2\times\{e,o,p_e,p_o\}$, forcing $0$ to be
    marked $e$, copying the $p$ marks from $B_2$ as $p_e$ or $p_o$,
    and forbidding patterns $(e,t,o)$, $(o,t,e)$, $(e,t,p_o)$,
    $(o,t,p_e)$, $(p_e,t,e)$ and $(p_o,t,o)$ in the new layer.
  \item Mark all integers of the form $2^n+2^m$, for $n>m$, as
    $q$. This is done by setting
    $B_4=B_3\times\{q,\mathvisiblespace\}$, and (just as we marked
    before the powers of $2$) marking by $q$ all integers having a
    different parity ($e/o$) than their $u^{-1}$-neighbour.
  \end{enumerate}

  The construction above clearly comes from a finite collection
  $\Theta_0\subset B_0\times\{t,u\}\times B_0$ of dominoes, which
  produces as unique colouring the specified marks.

  The vertices marked by a $q$ form a grid, more precisely an octant
  $\{(n,m):n>m\}$, with $(n,m)$ represented as $2^n+2^m$. It remains
  to show how the neighbourhood relation in this octant can be
  realized. The edge between $(n,m)$ and $(n\pm1,m)$ is realized by:
  starting from a $q$-marked vertex, follow $u^{\pm1}$ to the next
  $q$-marked vertex. The edge between $(n,m)$ and $(n,m\pm1)$ is
  realized by: starting from a $q$-marked vertex, follow $t^{\pm1}$ to
  the next $q$-marked vertex.
\end{proof}

The reduction of the tiling problem on the octant to the seeded tiling
problem on $\Gamma$ may be seen quite explicitly as follows: given an
instance $\Theta\subseteq B\times\{S,E,N,W\}\times B$ of the domino
problem on the octant, construct a multi-layered domino problem on
$\Gamma$: each vertex has one colour in $B_0$ and two colours in $B$,
one for the vertical direction and one for the horizontal one. The
domino rules impose that the colour in $B_0$ marks $q$-vertices as
above; that non-marked vertices propagate their horizontal colour
along $u$-edges and their vertical colour along $t$-edges; and that
$q$-marked vertices check that the propagated vertical and horizontal
colours match $\Theta$.

Note that the octant does not have vertices $(n,n)$, so there is no
vertical edge from $(n,n-1)$ to $(n,n)$. The domino tiling on the
octant ignores the $N$ direction at $(n-1,n)$, at likewise the domino
rules propagating colours vertically may be required to ignore their
constraint as they cross through a vertex marked $p$.

The graph $\Gamma$ is highly intransitive: the origin $0$ is the
unique vertex having a loop labelled $u$. We use this feature to
distinguish it among all other vertices:
\begin{prop}\label{prop:lrssu}
  The sunny-side-up $S_0$ is sofic on $\Gamma$.
\end{prop}
\begin{proof}
  Choose as colours $B=\{0,-_0,-_1,+_0,+_1\}$, and consider the tileset
  \[\Theta=\{(0,u,0),(0,t,+_*),(+_*,t,+_*),(-_*,t,-_*),(-_*,t,0),(*_0,u,*_1),(*_1,u,*_0)\}.\]
  Then at most one vertex of $\Gamma$ may be coloured $0$: all its
  neighbours in the $t$ direction are coloured $\{+_0,+_1\}$, and all
  its neighbours in the $t^{-1}$-direction are coloured
  $\{-_0,-_1\}$. This vertex coloured $0$, if it exists, must be the
  origin because its $u$-neighbour is also coloured $0$. On the other
  hand, the origin has to be coloured $0$, since all $u$-edges have
  distinct colours at their extremities if they're not $0$.

  It remains to see that $X_\Theta$ is not empty. The colouring in
  which $0$ is coloured $0$ and $2^s(2m+1)$ is coloured
  $\operatorname{sign}(m-\frac12)_{m\bmod 2}$ is legal.
\end{proof}

\noindent We are ready to put the pieces together:
\begin{proof}[Proof of Theorem~\ref{thm:longrange}]
  We invoke the classical fact that the domino tiling is undecidable
  on the octant~\cite{wang:ptpr2}. By Proposition~\ref{prop:lrsimul}
  the seeded domino problem is undecidable on $\Gamma$. By
  Proposition~\ref{prop:lrssu} and Lemma~\ref{lem:ssu} the domino
  problem is equivalent to the seeded domino problem on $\Gamma$.
\end{proof}

If the reader got the impression that the domino problem was deduced
from the seeded domino problem using a cheap trick, it's because it's
the truth: the orbit $G\cdot 0^\infty$ is special in that its position
$0^\infty$ is marked by a $u$-loop. A more sophisticated argument would be
required to study the domino problem on other orbits.

\subsection{The Barbieri-Sablik $H$-graph}\label{ss:bs}
Our second example is a variant of a graph considered by
Barbieri-Sablik~\cite{barbieri-sablik:ssdomino} in their investigation
of domino problems on self-similar structures,
see~\S\ref{ss:bsss}. Consider the graph with vertex set $\N\times\Z$,
and two kinds of edges: connecting $(m,n)$ with $(m,n\pm1)$; and
connecting
$(2^s m-1,2^s n)$ and $(2^s m,2^s n)$ for all $s\ge0,m\ge1,n\in\Z$.\\
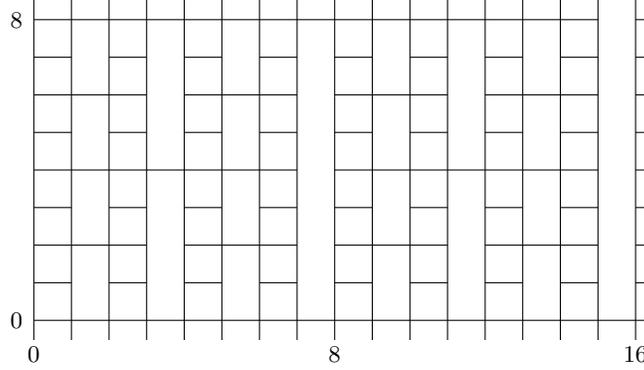
\centerline{\begin{tikzpicture}[scale=0.5] \node[label={left:\small
      $0$}] at (0.2,0) {}; \node[label={left:\small $8$}] at (0.2,8)
    {}; \node[label={below:\small $0$}] at (0,-0.2) {};
    \node[label={below:\small $8$}] at (8,-0.2) {};
    \node[label={below:\small $16$}] at (16,-0.2) {}; \clip
    (-0.5,-0.5) rectangle (16.5,8.5); \foreach\i in {0,...,16} \draw
    (\i,-1) -- (\i,9); \foreach\s in {1,2,4,8,16} {
      \pgfmathsetmacro\threes{3*\s} \foreach\i in {\s,\threes,...,17}
      { \foreach\j in {0,\s,...,8} \draw (\i-1,\j) -- +(1,0); } }
  \end{tikzpicture}}

This graph is also the Schreier graph of a transducer group. In the
transducer below, the alphabet $S$ is $\{00,01,10,11\}$ and the $?$
symbols stand for a wild card:\\
\centerline{\begin{fsa}[baseline]
    \node[state] (y) at (0,0) {$y$};
    \node[state] (x) at (2.5,0) {$x$};
    \node[state] (e) at (5,0) {$e$};
    \node[state] (z) at (7.5,0) {$z$};
    \path (y) edge[out=150,in=210,loop] node[left] {$00|00$} ()
    (y) edge node[below] {$10|10$} (x)
    (y) edge[bend left] node {$?1|?1$} (e)
    (x) edge node[below] {$1?|0?, 0?|1?$} (e)
    (z) edge[out=-30,in=30,loop] node[right] {$?1|?0$} ()
    (z) edge node {$?0|?1$} (e);
  \end{fsa}}

There is a natural bijection between $\N\times\Z$ and infinite words
in $S^\N$ ending in $(00)^\infty$ or $(01)^\infty$: firstly,
$S=\{0,1\}\times\{0,1\}$ so every sequence $w\in S^\N$ corresponds to
a pair of sequences $(u,v)\in\{0,1\}^\N$. Now given
$(m,n)\in\N\times\Z$, the pair of sequences representing it is $(u,v)$
with $u$ the Gray encoding of $m$ and $v$ the binary encoding of
$n$. Recall that the binary encoding of $n=\sum_{i\ge0}v_i2^i$, with
almost all $v_i=0$ or almost all $v_i=1$, is $v_0 v_1\dots$; and that
the Gray encoding of $m=\sum_{i\ge0}u_i2^i$ is
$(u_0\oplus u_1)(u_1\oplus u_2)\dots$ with $\oplus$ the exclusive-or
of bits $\in\{0,1\}$. It is then clear that generator $z$ connects
$(u,1^s0v_{s+1}\dots)$ to $(u,0^s1v_{s+1}\dots)$ and therefore $(m,n)$
to $(m,n+1)$; that generator $x$ has order $2$ and connects
$(0u_1\dots,v)$ with $(1u_1\dots,v)$ and therefore $(2m,n)$ with
$(2m+1,n)$; and that generator $y$ also has order $2$ and connects
$(0^s10u_{s+2}\dots,0^{s+1}v_{s+2}\dots)$ with
$(0^s11u_{s+2}\dots,0^{s+1}v_{s+2}\dots)$ and therefore
$(2^s(2m+1)-1,2^{s+1}n)$ with $(2^s(2m+1),2^{s+1}n)$. This concludes
the proof that the graph $\Gamma$ is indeed the Schreier graph of the
group $\langle x,y,z\rangle$. Set $N=\{x,y,z\}$, the nucleus of the
group generated by $\{x,y,z\}$.

We shall first show that $\Gamma$ simulates the ``hyperbolic
horoball''. This is the graph $\Delta$ with vertex set
$\{(2^{s+1}-1,2^s n):s\ge0,n\in\Z\}$, and with edges between
$(2^{s+1}-1,2^s(2n))$ and $(2^{s+2}-1,2^{s+1}n)$ and between
$(2^{s+1}-1,2^s n)$ and $(2^{s+1}-1,2^s(n+1))$. It corresponds (after
$90^\circ$ rotation) to a tiling of the horoball $\{\Im(z)\ge1\}$ in the
hyperbolic plane by pentagons with euclidean-straight sides and angles
$(90^\circ,90^\circ,90^\circ,90^\circ,180^\circ)$.

\begin{prop}\label{prop:bssimul}
  The graph $\Gamma$, when rooted at $(0,0)$, simulates the hyperbolic
  horoball.
\end{prop}
It is drawn below in thick lines, on top of $\Gamma$ (for which the
generators $x,y,z$ are
drawn respectively in red, green, blue):\\
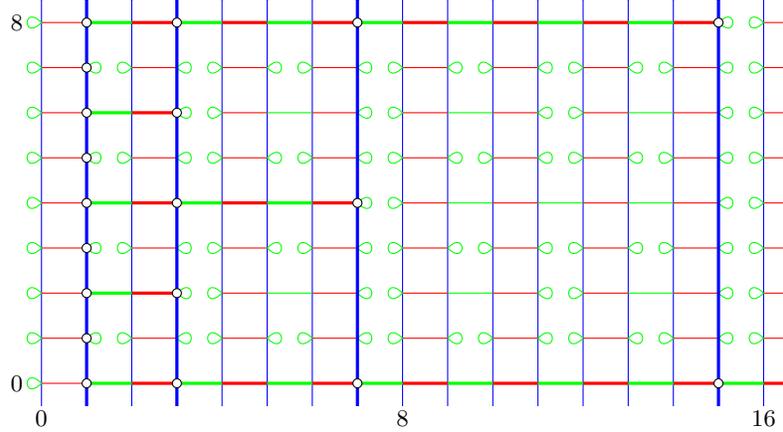
\centerline{\begin{tikzpicture}[scale=0.6] \node[label={left:\small
      $0$}] at (0,0) {}; \node[label={left:\small $8$}] at (0,8) {};
    \node[label={below:\small $0$}] at (0,-0.2) {};
    \node[label={below:\small $8$}] at (8,-0.2) {};
    \node[label={below:\small $16$}] at (16,-0.2) {};
    \clip (-0.5,-0.5) rectangle (16.5,8.5);
    \foreach\i in {0,...,16} \draw[blue,thin] (\i,-1) -- (\i,9);
    \foreach\i in {0,2,...,16} { \foreach\j in {0,...,8} \draw[red,thin] (\i,\j) -- +(1,0); }
    \foreach\s in {2,4,8,16} {
      \pgfmathsetmacro\threes{3*\s}
      \foreach\i in {\s,\threes,...,16} { \foreach\j in {0,\s,...,8} \draw[green,thin] (\i-1,\j) -- +(1,0); }
    }
    \foreach\i in {0,...,16} {
      \draw[green] (0,\i) .. controls +(135:0.6) and +(225:0.6) .. +(0,0);
    }
    \foreach\s in {2,4,8,16} {
      \foreach\i in {0,\s,...,16} {
        \foreach\j in {0,\s,...,16} {
          \draw[green] (\i+\s-1,\j+0.5*\s) .. controls +(45:0.6) and +(-45:0.6) .. +(0,0);
          \draw[green] (\i+\s,\j+0.5*\s) .. controls +(135:0.6) and +(225:0.6) .. +(0,0);
        }
      }
    }
    \foreach\i in {1,3,7,15} \draw[very thick,blue] (\i,-1) -- (\i,9);
    \foreach\i/\j in {0/18,2/4,4/8,6/4,8/16} {
      \foreach\k in {0,2,...,\j} {
        \ifnum\k>2
        \draw[very thick,red] (\k-2,\i) -- +(1,0);
        \draw[very thick,green] (\k-3,\i) -- +(1,0);
        \fi
      }
    }
    \foreach\s in {1,2,4,8} {
      \foreach\i in {0,...,16} {
        \filldraw[fill=white] (2*\s-1,\i*\s) circle (1mm);
      }
    }
  \end{tikzpicture}}

\begin{proof}
  The representation of $\Delta$ as a subset of $\N\times\Z$ matches
  that of $\Gamma$, and we shall construct a domino problem whose
  unique solution is a marking that singles out the vertices of
  $\Delta$ along with its edges.  Consider the set of colours
  \[B_0=\{a_0,a_1,b_0,\dots,b_4,c_0,c_1,c_2,d_0,d_1,d_2\}\]
  the generating set $S=\{x,y,z\}$ and the tileset
  \[\Theta_0=\left\{\begin{matrix}
        (a_0,x,b_0)&(b_0,x,c_0)&(c_0,x,d_0)&(a_1,x,b_1)&(a_1,x,b_2)\\
        (b_1,x,c_1)&(b_2,x,c_1)&(c_1,x,d_1)&(b_3,x,c_2)&(b_4,x,c_2)\\
        (c_2,x,d_2)\\
        (a_0,y,a_0)&(b_0,y,c_0)&(c_0,y,d_0)&(a_1,y,a_1)&(b_1,y,c_1)\\
        (b_2,y,b_2)&(c_1,y,d_1)&(b_3,y,b_3)&(b_4,y,b_4)&(c_2,y,c_2)\\
        (c_2,y,d_2)&(d_2,y,d_2)\\
        (a_0,z,a_1)&(a_1,z,a_0)&(a_1,z,a_1)&(b_0,z,b_2)&(b_0,z,b_3)\\
        (b_1,z,b_2)&(b_1,z,b_3)&(b_2,z,b_0)&(b_2,z,b_1)&(b_2,z,b_4)\\
        (b_3,z,b_2)&(b_3,z,b_3)&(b_4,z,b_0)&(b_4,z,b_1)&(b_4,z,b_4)\\
        (c_0,z,c_2)&(c_1,z,c_2)&(c_2,z,c_0)&(c_2,z,c_1)&(c_2,z,c_2)\\
        (d_0,z,d_2)&(d_1,z,d_2)&(d_2,z,d_0)&(d_2,z,d_1)&(d_2,z,d_2)
      \end{matrix}\right\},
  \]
  more conveniently given as the edges in the following graph:\\
  
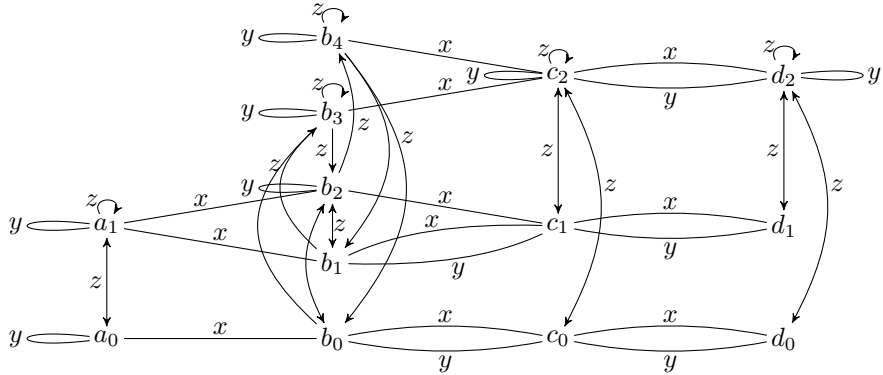
\centerline{\begin{tikzpicture}[xscale=1.5,yscale=0.5,every node/.style={inner sep=0.5mm},every loop/.style={},>=stealth']
    \node (a0) at (0,0) {$a_0$};
    \node (b0) at (2,0) {$b_0$};
    \node (c0) at (4,0) {$c_0$};
    \node (d0) at (6,0) {$d_0$};
    \node (a1) at (0,3) {$a_1$};
    \node (b1) at (2,2) {$b_1$};
    \node (b2) at (2,4) {$b_2$};
    \node (c1) at (4,3) {$c_1$};
    \node (d1) at (6,3) {$d_1$};
    \node (b3) at (2,6) {$b_3$};
    \node (b4) at (2,8) {$b_4$};
    \node (c2) at (4,7) {$c_2$};
    \node (d2) at (6,7) {$d_2$};
    \draw (a0) edge[loop left,-] node[left] {$y$} () edge node[above] {$x$} (b0)
    (b0) edge[bend left] node[above] {$x$} (c0) edge[bend right] node[below] {$y$} (c0)
    (c0) edge[bend left] node[above] {$x$} (d0) edge[bend right] node[below] {$y$} (d0)
    (a1) edge[loop left,-] node[left] {$y$} () edge node[above] {$x$} (b1) edge node[above,pos=0.4] {$x$} (b2)
    (b1) edge[bend left] node[above] {$x$} (c1) edge[bend right] node[below] {$y$} (c1)
    (b2) edge node[above] {$x$} (c1) edge[loop left] node[left] {$y$} ()
    (c1) edge[bend left] node[above] {$x$} (d1) edge[bend right] node[below] {$y$} (d1)
    (b3) edge node[above] {$x$} (c2) edge[loop left] node[left] {$y$} ()
    (b4) edge node[above] {$x$} (c2) edge[loop left] node[left] {$y$} ()
    (c2) edge[loop left] node[left] {$y$} () edge[bend left] node[above] {$x$} (d2) edge[bend right] node[below] {$y$} (d2)
    (d2) edge[loop right] node[right] {$y$} ()
    (a0) edge[<->] node[left] {$z$} (a1)
    (a1) edge[loop above,->] node[left=0.5mm] {$z$} ()
    (b0) edge[bend left=10,<->] (b2) edge[bend left=20] (b3)
    (b1) edge[<->] node[right=-0.5mm] {$z$} (b2) edge[bend left=20,->] node[left,pos=0.7] {$z$} (b3)
    (b2) edge[->,bend right=8] node[right,pos=0.4] {$z$} (b4)
    (b3) edge[loop above,->] node[left=0.5mm] {$z$} () edge[->] node[left] {$z$} (b2)
    (b4) edge[loop above,->] node[left=0.5mm] {$z$} () edge[bend left=15,->] (b1) edge[bend left=15,->] node[right,pos=0.3] {$z$} (b0)
    (c0) edge[bend right=10,<->] node[right,pos=0.6] {$z$} (c2)
    (c1) edge[<->] node[left] {$z$} (c2)
    (c2) edge[loop above,->] node[left=0.5mm] {$z$} ()
    (d0) edge[bend right=10,<->] node[right,pos=0.6] {$z$} (d2)
    (d1) edge[<->] node[left] {$z$} (d2)
    (d2) edge[loop above,->] node[left=0.5mm] {$z$} ();
    \end{tikzpicture}}

  I claim that, if the origin $(0,0)$ is coloured $a_0$ in a colouring
  respecting $\Theta_0$, then the nodes coloured $b_0,b_1,b_2$
  constitute the vertices of the hyperbolic horoball $\Delta$; that
  horizontal edges are marked by following an arbitrary power of $x y$
  over vertices coloured $\{c_1,d_1\}$ till the next vertex of
  $\Delta$ is reached; and that vertical edges are marked by following
  an arbitrary power of $z$ over vertices coloured $\{b_3,b_4\}$ till
  the next vertex of $\Delta$ is reached.

  These claims follow from a series of ``Sudoku'' deductions. The
  rules first imply that the horizontal axis is coloured
  $a_0((b_0|d_0)c_0)^\infty$, and that all the colours on a vertical
  line share the same letter $(a,b,c,d)$. If any vertex along the
  vertical axis is coloured $a_0$, then the corresponding row must
  also be coloured $a_0((b_0|d_0)c_0)^\infty$, so in particular it
  cannot contain any $y$ loop. However, all rows at height $\ne0$
  contain such a $y$ loop, so the vertical axis must be entirely
  coloured $a_1$, except for the origin.

  Now if some vertex on the vertical axis is coloured $a_1$, then the
  corresponding row is coloured
  $a_1((b_1|d_1)c_1)^*b_2c_2((d_2c_2)^*(b_3|b_4))^\infty$ with a $b_2$
  at the position of the first $y$ loop. Since this first loop appears
  at abscissa $2^{s+1}-1$ for some $s\ge0$, the columns
  $\{2^{s+1}-1\}\times\Z$ are all coloured using
  $\{b_0,\dots,b_4\}$. In the upwards direction, they are coloured
  $b_0(b_3^*b_2b_4^*b_1)^\infty$, and symmetrically
  downwards. Combining the conditions on the rows and columns, the
  colour of vertex $(2^{s+1}-1,2^s n)$ is $b_0$ if $n=0$, is $b_1$ if
  $n$ is even, and is $b_2$ if $n$ is odd.

  If furthermore $s>0$, then along that column $\{2^{s+1}-1\}\times\Z$
  there are vertices coloured $b_3$ and $b_4$, at
  $(2^{s+1}-1,2^{s-1}n)$ with $n$ odd. These are connected by a
  sequence of $x$ and $y$ to $(2^s,2^{s-1}n)$, along a horizontal
  segment coloured using $\{c_2,d_2\}$; so all columns
  $\{2m-1\}\times\Z$ with $m>0$ not a power of $2$ are coloured using
  $\{d_0,d_1,d_2\}$.

  The colouring is thus entirely specified by the dominoes, and the
  hyperbolic horoball appears exactly at the claimed position. Its
  vertices are connected, vertically, by sequences of $b_3$ or $b_4$,
  and horizontally by sequences of $c_1d_1$:\\
  
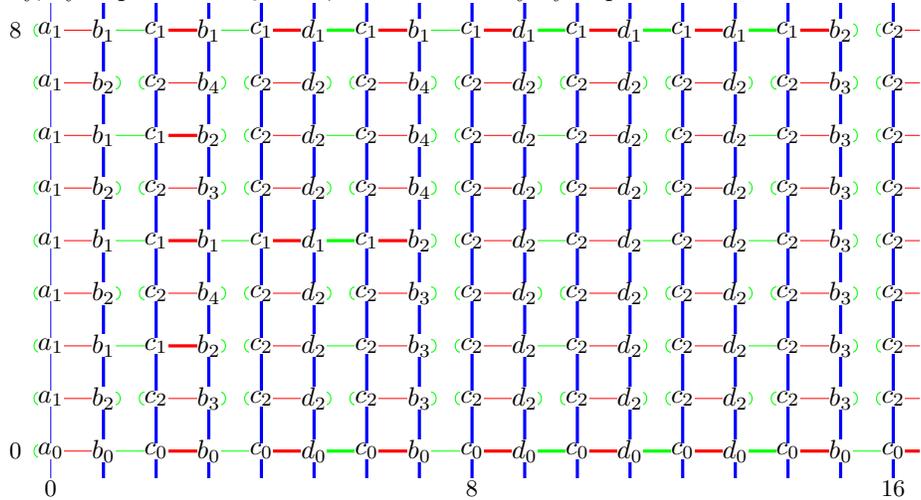
\centerline{\begin{tikzpicture}[scale=0.7,cnode/.style={inner sep=0mm,fill=white}]
    \node[label={left:\small $0$}] at (-0.2,0) {};
    \node[label={left:\small $8$}] at (-0.2,8) {};
    \node[label={below:\small $0$}] at (0,-0.2) {};
    \node[label={below:\small $8$}] at (8,-0.2) {};
    \node[label={below:\small $16$}] at (16,-0.2) {};
    \clip (-0.5,-0.5) rectangle (16.5,8.5);
    \foreach\i in {0,...,16} \draw[blue,very thin] (\i,-1) -- (\i,9);
    \foreach\i in {0,2,...,16} { \foreach\j in {0,...,8} \draw[red,very thin] (\i,\j) -- +(1,0); }
    \foreach\s in {2,4,8,16} {
      \pgfmathsetmacro\threes{3*\s}
      \foreach\i in {\s,\threes,...,16} { \foreach\j in {0,\s,...,8} \draw[green,very thin] (\i-1,\j) -- +(1,0); }
    }
    \foreach\i in {0,...,16} {
      \draw[green] (0,\i) .. controls +(135:0.6) and +(225:0.6) .. +(0,0);
    }
    \foreach\s in {2,4,8,16} {
      \foreach\i in {0,\s,...,16} {
        \foreach\j in {0,\s,...,16} {
          \draw[green] (\i+\s-1,\j+0.5*\s) .. controls +(45:0.6) and +(-45:0.6) .. +(0,0);
          \draw[green] (\i+\s,\j+0.5*\s) .. controls +(135:0.6) and +(225:0.6) .. +(0,0);
        }
      }
    }
    \foreach\i in {1,...,16} \draw[very thick,blue] (\i,-1)  -- (\i,9);
    
    \foreach\i/\j in {0/18,2/4,4/8,6/4,8/16} {
      \foreach\k in {0,2,...,\j} {
        \ifnum\k>2\draw[very thick,red] (\k-2,\i) -- +(1,0);\fi
      }
    }
    \foreach\i/\j in {5/0,5/4,5/8,9/0,9/8,11/0,11/8,13/0,13/8} { \draw[very thick,green] (\i,\j) -- +(1,0); }
    
    \foreach\i in {0,...,16} {
      \ifnum\i=0\def\9{a}\else\ifodd\i\def\9{d}\else\def\9{c}\fi\fi
      \foreach\s in {1,3,7,15} { \ifnum\s=\i\global\def\9{b}\fi }
      \foreach\j in {0,...,8} {
        \ifnum\j=0\tmpcnta=0\else\ifnum\i=0\tmpcnta=1\else\tmpcnta=2\fi\fi
        \foreach\a/\b in {3/2,7/4,3/6,15/8} {
          \ifnum\a>\i\ifnum\b=\j\global\tmpcnta=1\fi\fi
        }
        \foreach\a/\b/\c/\d in {3/0/2/3,3/2/4/4,3/4/6/3,3/6/8/4,7/0/4/3,7/4/8/4,15/0/8/3} {
          \ifnum\a=\i\ifnum\b<\j\ifnum\c>\j\global\tmpcnta=\d\fi\fi\fi
        }
        \global\def\8{}
        \ifx b\9\ifnum\tmpcnta<3\global\def\8{\boldsymbol}\fi\fi
        \node[cnode] at (\i,\j) {$\8\9_{\8\the\tmpcnta}$};
      }
    }
    \end{tikzpicture}}
\end{proof}

Since the hyperbolic grid has undecidable domino
problem~\cite{kari:undecidabilitytp}, so does the hyperbolic horoball;
thus the seeded tiling problem is undecidable on $\Gamma$. If fact, we
shall not need this, and rather reduce to the tiling problem on
arbitrarily large strips.

Indeed the hyperbolic horoball separates strips of width $2^s$ between
$2^s-1$ and $2^{s+1}-1$. These are the vertices coloured
$b_0,\dots,b_4,c_0,c_1,d_0,d_1$ in the colouring by
$\Theta_0$. Movement in this grid is defined exactly as in the
hyperbolic horoball: follow either of $x,y,z$ till a new marked vertex
is reached.

Given an instance of the tiling problem for the plane, namely a set of
Wang tiles $\Theta_1\subseteq C^{\{S,E,N,W\}}$ an instance of the
seeded tiling problem on $\Gamma$ may be constructed as follows. The
colours are $B_0\times\Theta_1$; the legal edge colourings enforce the
rules $\Theta_0$ on the first co\"ordinate; propagate the second
co\"ordinate in the $z$ direction till a marked vertex is reached, at
which point the $N/S$ matching rules of $\Theta_1$ are imposed; and
impose the $E/W$ matching rules of $\Theta_1$ along $x$- and
$y$-coloured edges:
\[\Theta=\left\{((p,\theta),g,(p',\theta')):
    \begin{array}{c}
      (p,g,p')\in\Theta_0,\\
      g=z\wedge p'\in\{c_2,d_2,b_3,b_4\}\implies\theta'=\theta,\\
      g=z\wedge p'\in\{c_0,c_1,b_0,b_1,b_2\}\implies(\theta')_S=\theta_N,\\
      g=x\wedge p'\in\{c_0,c_1\}\implies(\theta')_E=\theta_W,\\
      g=y\wedge p'\in\{d_0,d_1\}\implies(\theta')_E=\theta_W.
    \end{array}\right\}.\]

It is then clear that there is a valid colouring of $\Gamma$ by
$\Theta$ with first co\"ordinate $b_0$ at the origin if and only if
arbitrarily long strips of the right Euclidean half-plane can be coloured by
$\Theta_1$.

\noindent Our next step would be to consider the sunny-side-up on
$\Gamma$. Unfortunately,
\begin{lem}
  The sunny-side-up on $\Gamma$ is not sofic.
\end{lem}
\begin{proof}
  Assume for contradiction that $S_{(0,0)}$ is sofic, and let
  $X_\Theta\to S_{(0,0)}$ be the factor map. Consider
  $\tau\in X_\Theta$, and define $\tau_s$ as the map $\N\times\Z\to B$
  given by $\tau_s(m,n)=\tau(m+2^s,n+2^{s-1})$. By compactness, the
  $\tau_s$ have an accumulation point $\tau_\infty$. Since the pointed
  labelled graphs $(\Gamma,(2^s,2^{s-1}))$ converge to the pointed
  graph $(\Gamma,(0,0))$, we have $\tau_\infty\in X_\Theta$, yet
  $\pi(\tau_\infty(0,0))\neq1$ since it is a pointwise limit of
  $(0,0,\dots)$.
\end{proof}

We follow a slightly different strategy to prove that the domino
problem is undecidable on $\Gamma$, which amounts to simulating the
hyperbolic horoball without marking a root on $\Gamma$.  We shall
construct, in steps, a tileset that admits as single tiling a specific
configuration of edge and vertex decorations in $\Gamma$, whose edges
form a graph quasi-isometric to the tiling of the hyperbolic horoball
by right-angled pentagons. We shall not directly write down the
tileset $\Theta$, but rather use more general patterns. This is of
course equivalent thanks to Lemma~\ref{lem:pattern}.

\begin{proof}[Marking some segments]
  Vertices of $\Gamma$ will be coloured using
  $\{\mathvisiblespace,1,2,3,4,4'\}$; the colours $1,2,3,4,4'$ will
  mark all vertices with $4$ neighbours, and $\mathvisiblespace$ will
  mark the vertices with $3$ neighbours.

  \begin{enumerate}
  \item We force every vertex marked by an integer ($1,2,3,4,4'$) to
    have $4$ neighbours. (This follows from Lemma~\ref{lem:localmark},
    and is easily done by forcing each integer-marked vertex to have a
    different colour at its $y$-neighbour.)
  \item By forbidding all patterns
    $\{1,z,x z,z^{-1}x z\}\to\{b,c\}^4$, we force every
    $\{v, z v,x z v,z^{-1}x z v\}$ to contain at least one integer
    mark.
  \item We restrict the allowed integer combinations on endpoints of
    $x$- and $y$-edges:
    \begin{gather*}
      (1,y,1),\quad(2,y,4),\quad(3,y,3),\quad(4,y,4')\\
      (3,x,4),\quad(4,x,4').
    \end{gather*}
    We also force every column to be coloured $\mathvisiblespace^\Z$
    or $(1\mathvisiblespace(2|4|4')\mathvisiblespace)^\Z$ or
    $(3\mathvisiblespace^+(4|4')\mathvisiblespace^+)^\Z$, with as
    usual for regular expressions $\mathvisiblespace^+$ denoting a
    sequence of at least one $\mathvisiblespace$ and $(4|4')$ denoting
    either $4$ or $4'$. (This is easily enforced with dominoes by
    giving a different secondary colour to all $\mathvisiblespace$'s
    in the expressions above.)
  \end{enumerate}

  By the second rule, every vertex in $\{2m,2m+1\}\times\{4n+1,4n+2\}$
  contains an integer; and the only vertices with $4$ neighbours among
  these are the $(4m+1,4n+2)$ and $(4m+2,4n+2)$. Furthermore, the
  $x$-neighbour of these vertices have $3$ neighbours; so
  $(4m+1,4n+2)$ and $(4m+2,4n+2)$ must be coloured $1$.

  The vertices $(1,4n)$ also have an $x$-neighbour with $3$
  neighbours, so they must be coloured $2$ because of the alternation
  $(1,2)$ on column $1$.

  The odd rows are then coloured $\mathvisiblespace^\N$; those at
  height $4n+2$ are coloured
  $(\mathvisiblespace\mathvisiblespace11)^\N$; those at height
  $2^{s+1}(2n+1)$ for some $s\ge1$ are coloured
  $(\mathvisiblespace2(44')^{2^s-2}4334(4'4)^{2^s-2}2\mathvisiblespace)^\N$;
  and the the horizontal axis is coloured
  $\mathvisiblespace(44')^\N$. This again easily follows from Sudoku
  deductions: because of the $y$-loop at $(2^{s+2}-1,2^{s+1}(2n+1))$,
  the colours must start by an expression of the form
  $\mathvisiblespace2(44')^*42\mathvisiblespace$ with an odd number of
  $334$ inserted in it, and because of the vertical alternation
  $(3,4|4')$ the $33$ may only occur at position $2^{s+1}-1$ and
  $2^{s+1}$:\\
  
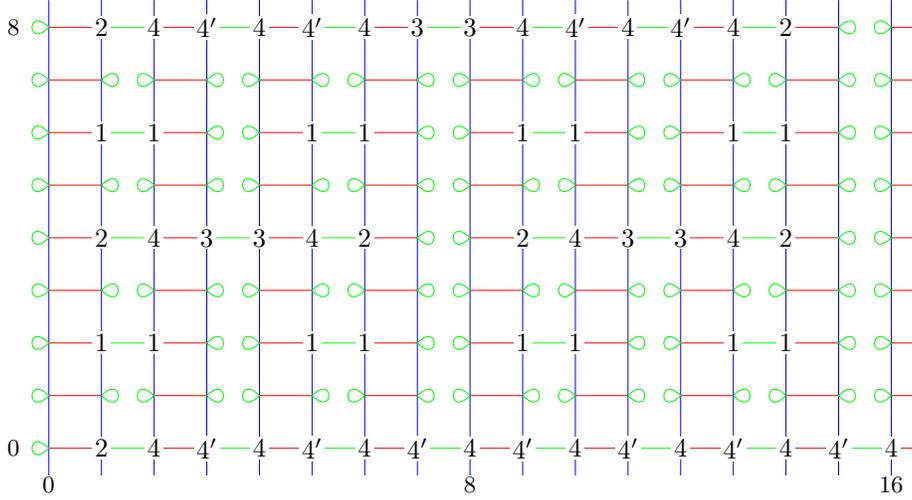
\centerline{\begin{tikzpicture}[scale=0.7,cnode/.style={inner sep=1pt,fill=white}]
    \node[label={left:\small $0$}] at (-0.2,0) {};
    \node[label={left:\small $8$}] at (-0.2,8) {};
    \node[label={below:\small $0$}] at (0,-0.2) {};
    \node[label={below:\small $8$}] at (8,-0.2) {};
    \node[label={below:\small $16$}] at (16,-0.2) {};
    \clip (-0.5,-0.5) rectangle (16.5,8.5);
    \foreach\i in {0,...,16} \draw[blue,very thin] (\i,-1) -- (\i,9);
    \foreach\i in {0,2,...,16} { \foreach\j in {0,...,8} \draw[red,very thin] (\i,\j) -- +(1,0); }
    \foreach\s in {2,4,8,16} {
      \pgfmathsetmacro\threes{3*\s}
      \foreach\i in {\s,\threes,...,16} { \foreach\j in {0,\s,...,8} \draw[green,very thin] (\i-1,\j) -- +(1,0); }
    }
    \foreach\i in {0,...,16} {
      \draw[green] (0,\i) .. controls +(135:0.6) and +(225:0.6) .. +(0,0);
    }
    \foreach\s in {2,4,8,16} {
      \foreach\i in {0,\s,...,16} {
        \foreach\j in {0,\s,...,16} {
          \draw[green] (\i+\s-1,\j+0.5*\s) .. controls +(45:0.6) and +(-45:0.6) .. +(0,0);
          \draw[green] (\i+\s,\j+0.5*\s) .. controls +(135:0.6) and +(225:0.6) .. +(0,0);
        }
      }
    }
    
    \foreach\i in {1,2,5,6,9,10,13,14} \foreach\j in {2,6} {
      \node[cnode] at (\i,\j) {$1$};
    }      
    \foreach\i/\j in {1/0,1/4,1/8,6/4,9/4,14/4,14/8} {
      \node[cnode] at (\i,\j) {$2$};
    }      
    \foreach\i/\j in {3/4,4/4,11/4,12/4,7/8,8/8} {
      \node[cnode] at (\i,\j) {$3$};
    }
    \foreach\i/\j/\k in {1/0/15,1/4/1,1/8/5,4/4/1,9/4/1,8/8/5,12/4/1} {
      \foreach\l in {1,...,\k} {
        \ifodd\l\def\9{}\else\def\9{'}\fi
        \node[cnode] at (\i+\l,\j) {$4\9$};
      }
    }
    \end{tikzpicture}}
\end{proof}

We now add orientations to some edges in $\Gamma$, to express flow of
information; distinguish certain vertices; and define some edges
between them.
\begin{proof}[Selecting vertices and edges]
  Vertices marked $1$ or $3$ always admit a $y$-neighbour with the
  same mark. The ``distinguished vertices'' we are interested in are
  the pairs of neighbours with identical odd mark. We call these
  \emph{clusters}, and more precisely $1$-clusters and $3$-clusters.

  We then add orientations to some edges. Every $1$-cluster has a
  single outgoing arrow (among its $6$ neighbours). This arrow must be
  start an oriented path of length $4$, labelled $z^{\pm2} y x$ along
  its edges, and going through vertices marked
  $1,\mathvisiblespace,2,4,3$. After the $z^{\pm2}$ it must run parallel with
  another path (coming from the $1$ at position $z^{\pm4}$ from the
  one it started at).

  Every $3$-cluster has two incoming double-arrows, along $x$-labelled
  edges, and a single outgoing arrow. As above, this arrow must start
  an oriented path labelled $(z^{\pm1})^*(y x)^*$ along its edges, and
  going through vertices marked
  $3,\mathvisiblespace,\dots,\mathvisiblespace,4',4,\dots,4',4,3$. After
  the $(z^{\pm1})^*$ it must run parallel with another path (coming
  from the $3$ at position $(z^{\pm1})^*$ from the one it started at).
  
  These sequences of arrows define ``imaginary'' paths between
  clusters. Additionally, the vertex from each cluster that does not
  have an outgoing arrow is connected by a ``real'' path
  labelled $(z^{\pm1})^*$ along its edges, and going through vertices
  marked
  $\{1,3\},\mathvisiblespace,\dots,\mathvisiblespace,4,\mathvisiblespace,\dots,\mathvisiblespace,\{1,3\}$. (The terminology ``real/imaginary'' comes from directions in the upper half-plane $\{\Im(z)>0\}$, that we shall use later.)

  We first claim that the above rules can be enforced by dominoes. To
  check this, it suffices to note that every edge gets up to two
  arrows, and that local rules determine the claimed tiling.

  We next claim that the rules enforce a unique, configuration of
  arrows. Indeed out of every four $1$-clusters at distance $4$ from
  each other, at least one can reach a unique $3$-cluster along
  $z^{\pm2}y x$; and this forces the three other $1$-clusters to be
  connected to the same $3$-cluster. The same argument applies to each
  of these $3$-clusters that were just connected to $1$-clusters: for
  every quadruple of such $3$-clusters, at least one of them will be
  connected to a unique $3$-cluster along $(z^{\pm1})^*(y x)^*$, and
  the other three in the quadruple will have to be be connected to the
  same higher-level $3$-cluster. This process can be carried on
  forever, resulting in a valid decoration of $\Gamma$ with arrows:\\
  
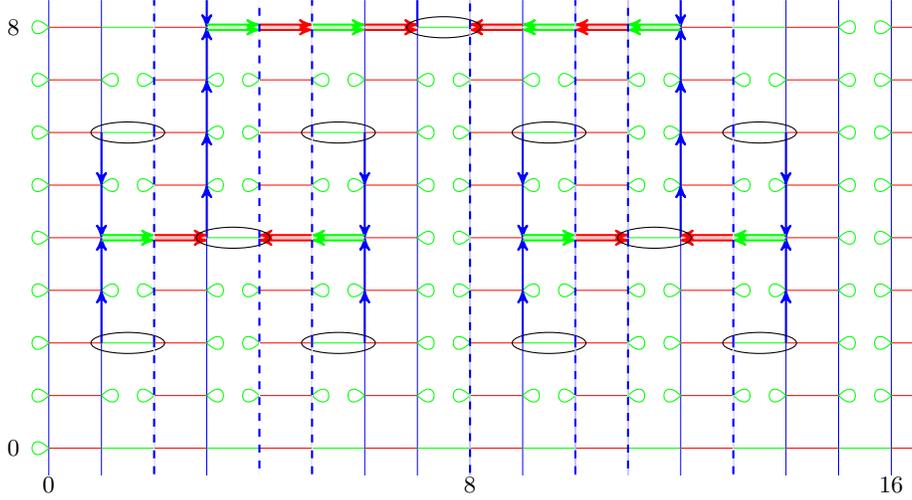
\centerline{\begin{tikzpicture}[scale=0.7,cnode/.style={inner sep=1pt,fill=white},>=stealth']
    \node[label={left:\small $0$}] at (-0.2,0) {};
    \node[label={left:\small $8$}] at (-0.2,8) {};
    \node[label={below:\small $0$}] at (0,-0.2) {};
    \node[label={below:\small $8$}] at (8,-0.2) {};
    \node[label={below:\small $16$}] at (16,-0.2) {};
    \clip (-0.5,-0.5) rectangle (16.5,8.5);
    \foreach\i in {0,...,16} \draw[blue,very thin] (\i,-1) -- (\i,9);
    \foreach\i in {0,2,...,16} { \foreach\j in {0,...,8} \draw[red,very thin] (\i,\j) -- +(1,0); }
    \foreach\i in {1,5,...,16} { \foreach\j in {0,2,...,8} \draw[green,very thin] (\i,\j) -- +(1,0); }
    \foreach\i in {3,11,...,16} { \foreach\j in {0,4,...,8} \draw[green,very thin] (\i,\j) -- +(1,0); }
    \foreach\i in {7} { \foreach\j in {0,8} \draw[green,very thin] (\i,\j) -- +(1,0); }
    \foreach\i in {15} { \foreach\j in {0} \draw[green,very thin] (\i,\j) -- +(1,0); }
    \foreach\i in {0,...,16} {
      \draw[green] (0,\i) .. controls +(135:0.6) and +(225:0.6) .. +(0,0);
    }
    \foreach\s in {2,4,8,16} {
      \foreach\i in {0,\s,...,16} {
        \foreach\j in {0,\s,...,16} {
          \draw[green] (\i+\s-1,\j+0.5*\s) .. controls +(45:0.6) and +(-45:0.6) .. +(0,0);
          \draw[green] (\i+\s,\j+0.5*\s) .. controls +(135:0.6) and +(225:0.6) .. +(0,0);
        }
      }
    }

    \foreach\i/\j/\d in {1/2/1,1/3/1,3/4/1,3/5/1,3/6/1,3/7/1,6/2/1,6/3/1,7/8/1,9/2/1,9/3/1,12/4/1,12/5/1,12/6/1,12/7/1,14/2/1,14/3/1,
      1/6/-1,1/5/-1,3/9/-1,6/6/-1,6/5/-1,9/6/-1,9/5/-1,12/9/-1,14/6/-1,14/5/-1} {
      \draw[->,thick,blue] (\i,\j) -- +(0,\d);
    }
    \foreach\i/\j/\d/\c in {1/4/1/green,2/4/1/red,6/4/-1/green,5/4/-1/red,3/8/1/green,4/8/1/red,5/8/1/green,6/8/1/red,9/4/1/green,10/4/1/red,14/4/-1/green,13/4/-1/red,12/8/-1/green,11/8/-1/red,10/8/-1/green,9/8/-1/red} {
      \draw[white] (\i,\j) -- +(\d,0);
      \draw[->,thick,\c] (\i,\j+0.05) -- +(\d,0);
      \draw[->,thick,\c] (\i,\j-0.05) -- +(\d,0);
    }
    \foreach\i/\j in {1.5/2,5.5/2,9.5/2,13.5/2,1.5/6,5.5/6,9.5/6,13.5/6,3.5/4,11.5/4,7.5/8} {
      \draw (\i,\j) ellipse (7mm and 2mm);
    }
    \foreach\i/\j/\d in {2/-2/4,2/2/4,2/6/4,4/-4/8,4/4/8,5/-2/4,5/2/4,5/6/4,8/-8/16,8/8/16,10/-2/4,10/2/4,10/6/4,11/-4/8,11/4/8,13/-2/4,13/2/4,13/6/4} {
      \draw[white] (\i,\j) -- +(0,\d);
      \draw[blue,dashed,thick] (\i,\j) -- +(0,\d);
    }
    \end{tikzpicture}}
\end{proof}

\noindent We are ready to put the pieces together:
\begin{proof}[Proof of Theorem~\ref{thm:bs}]
  Consider the graph with vertex set all clusters, connected by
  ``real'' edges (dashed above) and ``imaginary'' edges
  (following arrows). Consider furthermore domino problems on that
  graph, on which we impose the additional constraint that, at every
  $3$-cluster, the left incoming double arrow carries the same pair of
  dominoes as the right incoming double arrow. This is of course the
  same as considering the domino problem on the quotient graph in
  which the left and right incoming double arrows are identified, and
  so are the subgraphs they originate from; thus in the image above
  columns $0\dots2$ and $7\dots5$ are identified, columns $0\dots6$
  and $15\dots9$ are identified, etc. The resulting graph is a tiling
  of hyperbolic horoball by triangles and squares:\\
  
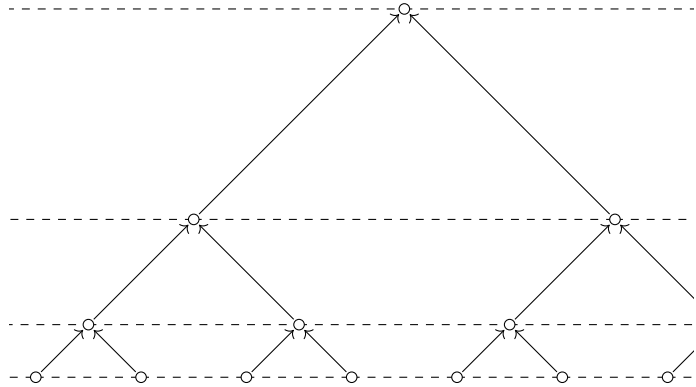
\centerline{\begin{tikzpicture}[scale=0.7]
      \clip (0.5,0.5) rectangle (13.5,8.5);
      \foreach\s in {1,2,4,8} {
        \pgfmathsetmacro\threes{3*\s}
        \foreach\i in {-\s,\s,...,14} {
          \draw[dashed] (\i,\s) -- +(2*\s,0);
          \ifnum\s>1
          \draw[->,shorten >=1mm,shorten <=1mm] (\i-\s/2,\s/2) -- +(\s/2,\s/2);
          \draw[->,shorten >=1mm,shorten <=1mm] (\i+\s/2,\s/2) -- +(-\s/2,\s/2);
          \fi
          \filldraw[fill=white] (\i,\s) circle (1mm);
        }
      }
    \end{tikzpicture}
  }    

  We invoke the classical fact that the domino tiling is undecidable
  on the hyperbolic plane~\cite{kari:revisited} (see
  also~\cite{margenstern:hyperbolicundecidable}). These sources
  typically consider the tiling of the hyperbolic plane by pentagons,
  but the problems are equivalent, since every pair of triangle and
  neighbouring square may be converted to a pentagon by ignoring their
  common edge.
\end{proof}

%%%%%%%%%%%%%%%%%%%%%%%%%%%%%%%%%%%%%%%%%%%%%%%%%%%%%%%%%%%%%%%%
\section{Barbieri-Sablik's self-similar structures}\label{ss:bsss}
In~\cite{barbieri-sablik:ssdomino}, the authors consider
substitutional colourings of Euclidean space, as means of defining
domino problems on self-similar sets. They consider a black/white
colouring $s$ of a $k_1\times\cdots\times k_d$ box in $\Z^d$, for
definiteness $\{0,\dots,k_1-1\}\times\cdots\times\{0,\dots,k_d-1\}$;
define $\lambda_0$ as a black box at the origin, and for all $n\ge1$
define $\lambda_n$ as the colouring of the
$k_1^n\times\cdots\times k_d^n$ box obtained by replacing, in
$\lambda_{n-1}$, every black box by $s$ and every white box by an
all-white box of size $k_1\times\cdots\times k_d$.  Let then
$\Lambda_s$ denote the generated $\Z^d$-subshift: it consists of all
maps $\tau\colon\Z^d\to\{\circ,\bullet\}$ such that, for every finite
subset $P\subset\Z^d$, the restriction of $\tau$ to $P$ coincides, up
to translation, with the restriction of some $\lambda_n$ to some
subset $P+x$ contained in the box $\lambda_n$.

Barbieri and Sablik then consider the following modification of the
domino problem on $\Z^d$, called ``$s$-domino problem'': an instance
is a set of colours $B=\{\circ\}\sqcup B_\bullet$ and tileset
$\Theta\subset B\times\{-1,0,1\}^d\times B$, assumed to contain
$(\circ,\varepsilon,\circ)$ for all $\varepsilon$; there is a natural
map $\pi\colon B\to\{\circ,\bullet\}$ given by
$B_\bullet\to\{\bullet\}$. Then the answer should be ``yes'' if and
only if there exists a non-trivial valid colouring of $\Z^d$ that
projects to $\Lambda_s$, namely a colouring $\tau\colon\Z^d\to B$ with
$\pi\circ\tau\in\Lambda_s$ and $\tau\not\in\{\circ\}^{\Z^d}$ and
$(\tau(x),\varepsilon,\tau(x+\varepsilon))\in\Theta$ for all
$x\in\Z^d,\varepsilon\in\{-1,0,1\}^d$. (Their formulation uses more
general patterns than dominoes, but this is equivalent by
Lemma~\ref{lem:pattern}).

Here are three examples of substitutions, all on the plane, taken from
their article:\\
\centerline{\bsrule{2}{1/0,0/0,1/1}\hspace{2cm}\bsrule{3}{0/0,0/1,0/2,1/1,2/0,2/1,2/2}\hspace{2cm}\bsrule{3}{0/0,0/1,0/2,1/0,1/2,2/0,2/1,2/2}}

Iteration of the first rule produces a discrete approximation of the
Sierpiński gasket, for which they prove that the domino problem is
decidable (as we will also see in~\S\ref{ss:sierpinski}). The third
rule produces a discrete approximation of the Sierpiński carpet, for
which they prove that the domino problem is undecidable. They list the
second example as an interesting border case between decidability and
undecidability.

In terms of shift spaces, an instance of the $s$-domino problem is a
set $\mathcal F$ of forbidden patterns on an alphabet
$B=\{\circ\}\sqcup B_\bullet$, leading to a natural map
$\pi\colon B^{\Z^d}\to\{\circ,\bullet\}^{\Z^d}$, and the question is
whether there is an $\mathcal F$-avoiding configuration in
$\pi^{-1}(\Lambda_s\setminus\{\circ\}^{\Z^d})$, namely if
$\pi^{-1}(\Lambda_s\setminus\{\circ\}^{\Z^d})\cap X_{\mathcal
  F}\neq\emptyset$. A variety of related problems may be asked, for
example ``does one have $\pi(X_{\mathcal
  F})\supseteq\Lambda_s$?''. The substitution $s$ induces a self-map
$\overline s$ of $\{\circ,\bullet\}^{\Z^d}$, replacing each $\bullet$
by the grid $s$ while preserving the origin. Set
$\Lambda'_s=\bigcap_{n\ge0}\overline s^n(\{\circ,\bullet\}^{\Z^d})$,
the set of configurations that admit infinitely many preimages under
$\overline s$. Then $\Lambda'_s\supseteq\Lambda_s$, and possibly
contains some extra ``limit'' configurations, such as the combination
of different elements of $\Lambda_s$ on different orthants. The above
questions ``$\pi(X_{\mathcal F})\subseteq\Lambda'_s$?''
``$\pi(X_{\mathcal F})\cap\Lambda'_s\ne\{\circ\}^{\Z^d}$?'' can also
be asked for $\Lambda'_s$.

Although in some cases these questions can have different answers, it
is possible, at least in all cases I considered, to reduce
decidability of one question to the other, so I will not devote too
much attention to these distinctions. For example, unless the
substitution $s$ is constant, the subshift $\Lambda_s$ is almost
minimal (it has a unique closed invariant subset $\{\circ\}^{\Z^d}$),
in which case $\pi(X_{\mathcal F})\cap\Lambda_s$ is either
$\{\circ\}^{\Z^d}$ or $\Lambda_s$ itself.

My point is, rather, that a wealth of interesting tiling problems
arise within the language of Schreier graphs, and that domino problems
defined via substitutions in Euclidean space can be reformulated in a
natural way by ridding them of an ambient space in which they embed.
In the case of \emph{contracting} groups, the Schreier graphs can in
principle be quasi-isometrically imbedded in $\R^d$ for some $d$, and
therefore could, in principle, be cast into a language of
substitutions. This doesn't even seem approachable for other examples
such as the long range graph.

\subsection{Self-similar structures}
Let us see how graphs naturally arise from the substitution $s$. Let
$M$ denote the diagonal matrix with entries $k_1,\dots,k_d$ along the
diagonal. We first associate with $s$ a self-similar structure in the
sense of Kigami. Let $S$ denote the set of coordinates $\in\Z^d$ at
which $s(\tikz{\fill(0,0) rectangle (1.5ex,1.5ex);})$ has a black box,
and for each $v\in S$ consider the affine map
$F_v(x)=M^{-1}(x+v)$. Start by $K_0=[0,1]^d$, and for each $n\ge0$ set
$K_{n+1}=\bigcup_{v\in S}F_v(K_n)$. Set finally
$K=\bigcap_{n\ge0}K_n$. Then $(K,\{F_v\}_{v\in S})$ is a self-similar
structure: the maps $F_v$ are contractions, so the coding map
$\pi\colon S^{-\N}\to K$ maps $(v_i)_{i\le0}$ to the limit of
$F_{v_0}(F_{v_{-1}}(F_{v_{-2}}(\cdots)))$.

The recipe of~\S\ref{ss:ss} then produces graphs $\Gamma_\xi$ for all
$\xi\in S^\N$; however we shall need \emph{labelled} graphs for the
domino problem, so we rather turn $s$ into a self-similar group.

In fact, the most natural algebraic structure to associate with a
substitution $s$ is a \emph{pseudo-group}, namely a collection of
partially-defined bijections of $S^\N$, closed under
composition. These bijections will be given by partially-defined maps
$\Phi\colon A\times S\dashrightarrow S\times A$. I will remark later
how to obtain \emph{bona fide} group actions.

The transducer $\Phi$ associated with the substitution $s$ has alphabet
$S$ and stateset $A=\{-1,0,1\}^d$. For each $v\in S,a\in A$: if
$v+a=v'+M(a')$ for some $v'\in S,a'\in A$ then $\Phi(a,v)=(v',a')$,
and otherwise $\Phi(a,v)$ is not defined. Note that the $v',a'$ above
are unique if they exist, since $M$ has full rank and no two elements
of $S$ are congruent modulo $M(\Z^d)$. In particular, state $0^d$ is
the identity. For example, the transducer associated with the Sierpiński gasket, with alphabet $\{0=(0,0),1=(1,0),2=(1,1)\}$, is\\
\centerline{\begin{fsa}[scale=1,every state/.style={minimum size=6mm,inner sep=0mm}]
    \node[state] (nw) at (-2,2) {$\nwarrow$};
    \node[state] (n) at (0,2) {$\uparrow$};
    \node[state] (ne) at (2,2) {$\nearrow$};
    \node[state] (w) at (-2,0) {$\leftarrow$};
    \node[state] (x) at (0,0) {$\cdot$};
    \node[state] (e) at (2,0) {$\rightarrow$};
    \node[state] (sw) at (-2,-2) {$\swarrow$};
    \node[state] (s) at (0,-2) {$\downarrow$};
    \node[state] (se) at (2,-2) {$\searrow$};
    \path (nw) edge node {$2|0$} (n) edge node {$0|2$} (w)
    (n) edge[loop right] node {$2|1$} () edge node {$1|2$} (x)
    (ne) edge[loop right] node {$2|0$} () edge node {$0|2$} (x)
    (w) edge[loop left] node {$0|1$} () edge node {$1|0$} (x)
    (e) edge[loop right] node {$1|0$} () edge node {$0|1$} (x)
    (sw) edge[loop left] node {$0|2$} () edge node {$2|0$} (x)
    (s) edge[loop left] node {$1|2$} () edge node {$2|1$} (x)
    (se) edge node {$0|2$} (s) edge node {$2|0$} (e);
  \end{fsa}}

Note that the transducer $\Phi$ is contracting, and that its nucleus
is a subset of $A$. However, it need not be recurrent, for example the
state $\searrow$ above is not.

To obtain a genuine, everywhere-defined action, we can now replace
every generator $a$ by a product $b\cdot c$ of two involutions, and
force $b$ or $c$ to have fixed points where $a$ is not defined. In
effect, we replace the orbits of $a$, which are lines or line
segments, by orbits of (finite or infinite) dihedral groups
$\langle b,c\rangle$. This adds notational complications without
changing much about the Schreier graphs. There is, however, the
important effect that we are adding loops to the Schreier graph, which
can then be detected by dominoes.

To relate the domino problems on $s$ and the Schreier graphs of
$\Phi$, fix a sequence $\xi=(\xi_n)\in S^\N$ of distinguished black
boxes. There is then a well-defined associated colouring of the plane:
start by the partial colouring $\lambda_1$ of $\Z^d$ with $\xi_1$ at
the origin; note that it extends to the partial colouring $\lambda_2$
with $\xi_1\xi_2$ (namely, the box $\xi_1$ inside the box $\xi_2$) at
the origin; and so on, with box $\xi_1\dots\xi_n$ of $\lambda_n$ at
the origin. If the limit does not colour all $\Z^d$, it means that
almost all $\xi_i$ lie on some boundary facet of the box $s$. Extend
then the colouring to $\Z^d$ by everywhere-$\circ$.

On the other hand, consider the Schreier graph of the partial action
of $\langle A\rangle$ on $\xi$, and call this graph $\Gamma_\xi$. The
following is an immediate translation:
\begin{prop}
  Assume that the collection of black squares in $s$ is connected, and
  consider $\xi\in S^\N$. Then the domino problem ``does there exist a
  colouring $\tau\colon\Z^d\to B$ projecting to $\tau_\xi$?'' is
  equivalent to the domino problem on $\Gamma_\xi$.
\end{prop}
\begin{proof}
  Since the black squares in $s$ are connected, the Schreier graph of
  the partial action of $\langle A\rangle$ on $S^n$ is connected for
  all $n$; so the orbit $\langle A\rangle\cdot\xi$ is naturally in
  bijection with the black boxes in $\tau_\xi$.

  The edges of $\Gamma_\xi$ are also naturally in bijection with the
  edges of the adjacency graph of black boxes in $\tau_\xi$. In
  considering the domino problem above $\tau_\xi$, there are also
  edges between white and black boxes: the tile at a black box can
  ``sense'' whether its neighbour in white, while the graph
  $\Gamma_\xi$ has no such edges. This is however easy to remedy by
  adding extra layers to the tiling: for example, to detect whether
  the left neighbour is white, add a colour $\ell$ with no domino
  $(*,\rightarrow,\ell)$, and force by local rules the colour $\ell$
  to appear as often as $s$ requires it.

  If almost all $\xi$ happen to lie on the same facet of $s$, then the
  Schreier graph $\Gamma_\xi$ is exceptional, and its vertices
  correspond to only an orthant of $\Z^d$. The tiling $\tau_\xi$,
  likewise, was extended by $\circ$ on the complement of the orthant.
\end{proof}

Even though there is a continuum of different Schreier graphs, varying
$\xi$, there are only finitely many different domino problems: the
collection of finite balls in the $\Gamma_\xi$ only depends on which
sides of the box $s$ contain all but finitely many of the
$\xi_n$'s. It follows that the ``substitutional domino problems'' of
Barbieri-Sablik and the graph domino problems $\Gamma_\xi$ are
essentially equivalent.

\subsection{The $H$-graph}
We will study in much more detail in~\S\ref{ss:bs} the middle example
of substitution above, here called the $H$-graph because of the shape
of its black boxes. Zooming at the central square produces a symmetric
figure, and the graph given in~\S\ref{ss:introu} is just half of the
picture. The domino problems are equivalent on the graph and its half,
and I chose to work with the half-graph because the transducer
producing it is simpler.

The recipe producing a transducer from a substitution, outlined above,
would have produced a partial action of $\langle \{-1,0,1\}^2\rangle$
on $\{1,\dots,7\}^\N$. Firstly, by simplifying the resulting graph, I
could put it more nicely in the plane, and use only $4$ symbols
instead of $7$. One of the generators --- the vertical translation
$(0,1)$ --- can be defined everywhere, and rewritten $z$; while the
other one, $(1,0)$, is replaced by two involutions $x,y$ whose orbits
generate dihedral groups of order $2^{s+2}$ on the row at height
$2^s(2n+1)$, and an infinite dihedral group on the horizontal axis.

It would also have been possible to extend the partial action into a
genuine group action, by letting, in the recursive formulas above,
$\Phi(a,v)=(v',a^\varepsilon)$ with $v'$ the first black box
following $v$ cyclically in direction $a$; and $\varepsilon=1$ if $v'$
is reached by wrapping around and $\varepsilon=0$ if not. We are then
adding edges to the graph $\Gamma_\xi$. Applying this recipe to the simplification of the $H$-graph produces the following transducer:\\
\centerline{\begin{fsa}[baseline] \node[state] (w) at (2.5,0) {$w$};
    \node[state] (e) at (5,0) {$e$}; \node[state] (z) at (7.5,0)
    {$z$}; \path (w) edge[out=150,in=210,loop] node[left] {$10|00$} ()
    (w) edge[bend right] node[below] {$11|01$} (e) (w) edge[bend left]
    node {$00|10,01|11$} (e) (z) edge[out=-30,in=30,loop] node[right]
    {$\begin{matrix}01|00\\11|10\end{matrix}$} () (z) edge node[below]
    {$10|11$} node[above] {$00|01$} (e);
  \end{fsa}}

In it, the edge `$11|01$' from $w$ to $e$ should be removed to define
a partial action of $w$; the Schreier graph then has vertices
naturally in bijection with $\Z^2$, via binary encoding, and is made
of two copies of the graph given in~\S\ref{ss:introu}, joined by an
edge. This changes nothing to the fundamental nature of the domino
problem.

\subsection{Substitutions with unbounded connectivity}\label{ss:isthmus}
In the next-to-last section of~\cite{barbieri-sablik:ssdomino}, the
authors propose a separation of substitutions in different classes,
using which they conjecturally settle the decidability of the domino
problem. Firstly, even if this is not explicit, they assume that the
collection $S$ of black squares in $s$ is connected. They distinguish
a set $\mathbb W$ of directions, which correspond to the recurrent
states of the transducer $\Phi$ above. They define then \emph{flexible
  lines in direction $t$} as sequences
$x_0,x_1,\dots,x_n=x_0+(0,\dots,k_t,\dots,0)$ in $S$ with
$x_j-x_{j-1}\in\mathbb W$ for all $j=1,\dots,n$; and say $s$ has
\emph{bounded connectivity} if there is at most one flexible line in
each direction, while $s$ has an \emph{isthmus} if there is one
flexible line in one direction, and at least two disjoint flexible
lines in another.

They claim that if substitution $s$ has bounded connectivity then the
$s$-domino problem is decidable; and that this follows from an
adaptation of their main theorem. Their definition of bounded
connectivity seems a bit too wide for this to hold; for example, the
substitution\\
\centerline{\bsrule{5}{0/0,0/1,0/2,0/3,1/2,2/2,3/2,4/2,4/3,4/4,1/3,1/4,2/4,2/0,3/0,3/1,4/1}}\\
may also simulate the hyperbolic grid, yet follows their definition of
``bounded connectivity''. What is true, and follows from their
Theorem~1 and from Theorem~\ref{thm:bounded}, is that if the
transducer $\Phi$ is bounded then the associated domino problem is
decidable.

Moreover, it could well be that $\Phi$ is not bounded, but a
\emph{conjugate} of $\Phi$ is bounded. Algebraically, this is just a
conjugate of the (pseudo)group by a transformation of $S^\N$, itself
given by an initial transducer. This may be phrased in the following
manner: a substitution is \emph{conjugate to bounded} if for every
direction $t\in\{1,\dots,d\}$ there is a partition of
$S\sqcup(S+(0,\dots,k_t,\dots,0))$ in two pieces $S_0\sqcup S_1$ such
that there is a single edge in direction $t$ connecting $S_0$ to
$S_1$. (The case of bounded transducers corresponds to $S_0=S$ and
$S_1=S+(0,\dots,k_t,\dots,0)$.)
\begin{prop}
  If $s$ is conjugate to bounded then the $s$-domino problem is decidable.
\end{prop}
\begin{proof}
  As the name hints, we shall show that the transducer $\Phi$ may be
  conjugated to a bounded transducer. We do this one direction at a
  time. In direction $a\in A$, let the partition be $S_0\sqcup
  S_1$. Define then the initial transducer $\Psi$, with alphabet $S$,
  stateset $\{\psi,a \psi\}$ and initial state $\psi$, by
  $\Psi(\psi,v)=(v,a^\varepsilon\psi)$ if $v+a\in S_\varepsilon$. Then
  the conjugate of $\Phi$ by $\psi$ is bounded in direction $a$.
\end{proof}

Now minimality of the dynamical system $\Lambda_s$ implies that, if
$s$ contains an isthmus, then its associated colourings contain
densely a (possibly deformed) copy of the $H$-substitution; or,
equivalently, the $H$-graph. Furthermore, local rules (purely based on
what appears inside $s$) allow this $H$-graph to be distinguished as a
quasi-isometrically imbedded subgraph. From this we deduce:
\begin{prop}
  If $s$ contains an isthmus then the $s$-domino problem is undecidable.\qed
\end{prop}

%%%%%%%%%%%%%%%%%%%%%%%%%%%%%%%%%%%%%%%%%%%%%%%%%%%%%%%%%%%%%%%%

\end{document}